\documentclass[12pt]{amsart}
\usepackage{amsfonts}
\usepackage{amsmath}
\usepackage{amssymb}
\usepackage[margin=1.2in]{geometry}
\usepackage{verbatim}
\newtheorem{theorem}{Theorem}[section]

\newtheorem{lemma}{Lemma}[section]
\newtheorem{definition}{Definition}[section]
\theoremstyle{remark}
\newtheorem{remark}[theorem]{Remark}

\usepackage{bm}

\newcommand{\wto}{\rightharpoonup}

\newcommand{\tq}{\tau_q}
\newcommand{\tqa}{\tq^\alpha}
\newcommand{\tT}{\tau_T}

\newcommand{\half}{\frac{1}{2}}
\newcommand{\nrm}[1]{\left\| #1 \right\|}
\newcommand{\abs}[1]{\left\lvert #1 \right\rvert}
\DeclareMathOperator*{\esssup}{ess\,sup}
\DeclareMathOperator{\argument}{arg}
\newcommand{\g}[1]{{g_{#1}}}
\newcommand{\ggamma}{\g{\gamma}}    
\newcommand{\galpha}{\g{\alpha}}

\newcommand{\ggammamin}{\g{\gamma-1}} 

\newcommand {\scal}[2]{\left(#1,#2\right)}

\newcommand{\inpX}[3]{\left\langle #1,#2 \right\rangle_{#3}}
\newcommand{\inpd}[2]{\inpX{#1}{#2}{{\hko{1}}^\ast \times \hko{1}}}

\DeclareMathOperator{\di}{d\hspace{-1.5pt}}

\newcommand {\dt}{ \di t}
\newcommand{\ds}{\di s}
\newcommand{\dtau}{\di \tau}

\newcommand {\CC } {{\mathbb C}} 
\newcommand {\NN } {{\mathbb N}}
\newcommand {\RR } {{\mathbb R}}
\newcommand{\veps}{\varepsilon}

\newcommand {\diver } {\nabla \cdot }

\newcommand {\pdt}{\partial_t}
\newcommand {\pdtt}{\partial_{tt}}

\newcommand {\I}{[0,T]}
\newcommand {\Iopen}{(0,T)}
\newcommand {\Iclosed}{[0,T]}

\newcommand {\domain}{\Omega}

\newcommand {\lp}[1]{\Leb^{#1} (\domain )}
\newcommand {\Lp}[1]{{\bf L}^{#1} (\domain )}

\DeclareMathOperator{\Leb}{L}

\DeclareMathOperator{\Cont}{C}
\DeclareMathOperator{\Hi}{H}

\newcommand {\hk}[1]{\Hi^{#1}(\domain )}
\newcommand {\hko}[1]{\Hi^{#1}_0(\domain )}

\newcommand {\lpIlp}{\Leb^2\left(\Iopen,\lp{2}\right)}

\newcommand {\lpkIX}[2]{\Leb^{#1}\left(\Iopen,#2\right)}

\newcommand {\cIX}[1]{\Cont\left(\I,#1\right)}

\def \vector#1{\mathbf{#1}}
\def \matrix#1{\mbox{\boldmath ${#1}$}}
\newcommand {\X}{\vector{x}}

\newcommand{\quand}{\quad \text{ and } \quad}

\usepackage{comment}
\usepackage[shortlabels]{enumitem}

\makeatletter
\@namedef{subjclassname@2020}{%
  \textup{2020} Mathematics Subject Classification}
\makeatother

\usepackage{xcolor}
\usepackage[hidelinks]{hyperref}
\hypersetup{
    colorlinks = true,
    linkcolor = red,
    anchorcolor = red,
    citecolor = green,
    filecolor = red,
    urlcolor = red,
    linktoc=page
    }
\begin{document}

	\title[Fractional SPL equation]{Existence and uniqueness of a weak solution to fractional single-phase-lag heat equation}
	
	\author[F.~Maes]{Frederick Maes$^1$}

	\author[K.~Van~Bockstal]{Karel Van Bockstal$^2$} 
	\thanks{The work of K.~Van Bockstal was supported by the Methusalem programme of Ghent University Special Research Fund (BOF) (Grant Number 01M01021)} 
	
	\address[1]{Research Group NaM$^2,$ Department of Electronics and Information Systems\\ Ghent University \\ Krijgslaan 281 \\ B 9000 Ghent \\ Belgium}
	\email{frederick.maes@UGent.be}
	\address[2]{Ghent Analysis \& PDE center, Department of Mathematics: Analysis, Logic and Discrete Mathematics\\ Ghent University\\
		Krijgslaan 281\\ B 9000 Ghent\\ Belgium} 
	\email{karel.vanbockstal@UGent.be}

	
	\begin{abstract} In this article, we study the existence and uniqueness of a weak solution to the fractional single-phase lag heat equation. This model contains the terms $\mathcal{D}_t^\alpha(\pdt u)$ and $\mathcal{D}_t^\alpha u $ (with $\alpha \in(0,1)$), where $\mathcal{D}_t^\alpha$ denotes the Caputo fractional derivative in time of constant order $\alpha\in(0,1)$. We consider homogeneous Dirichlet boundary data for the temperature. We rigorously show the existence of a unique weak solution under low regularity assumptions on the data. Our main strategy is to use the variational formulation and a semidiscretisation in time based on Rothe's method. We obtain a priori estimates on the discrete solutions and show convergence of the Rothe functions to a weak solution. The variational approach is employed to show the uniqueness of this weak solution to the problem. We also consider the one-dimensional problem and derive a representation formula for the solution. We establish bounds on this explicit solution and its time derivative by extending properties of the multinomial Mittag-Leffler function.  
	\end{abstract}
	
	\maketitle
	
	\tableofcontents
	\section{Introduction}\label{sec_intro}
	\subsection{Problem formulation} \label{subsec_Problem}
	Consider a material contained in a bounded domain $\domain\subset \mathbb{R}^d, d\in\mathbb{N}$ with a Lipschitz continuous boundary $\partial \domain.$ Set $Q_T :=\domain \times (0,T]$ and $\Sigma_T:= \Gamma \times (0,T],$ where $T >0$ denotes the final time. The function $u(\X,t)$ represents the temperature at a material point $\X \in \domain$ at time $t.$ The heat flux is denoted by $\mathbf{q}(\X,t).$ Let $\rho$ and $c$ be positive constants denoting the material's density and specific heat, respectively. Moreover, let $\matrix{k}$ be the thermal conductivity, which may be space-dependent.
	
	In \cite{Tzou1995}, a modification of the classical Fourier law for heat conduction processes,
	\begin{equation}
		\label{eq:tzoumodel_simp}
		\mathbf{q}(\X,t+\tq) = -\matrix{k}(\X) \nabla u (\X, t+\tT),
	\end{equation}
	was proposed to overcome the presence of infinite speed of heat propagation. This model links the heat flux $\mathbf{q}$ to the temperature gradient $\nabla u$ at the spatial point $\X$ at different times, allowing for a delay in the build-up of the heat flux or temperature gradient. Here $\tq$ and $\tT$ are the phase-lag parameters. These parameters are material properties and represent the relaxation and delay time, respectively. In case $\tq>0$ and $\tT = 0,$ the law is of \emph{single-phase-lag} (SPL) type. If $\tq>0$ and $\tT>0$, the law is called to be of \emph{dual-phase-lag} (DPL) type. The classical Fourier law is recovered if both lagging parameters equal zero. 
	
	In this contribution, we consider the model obtained by first-order Taylor approximation of \eqref{eq:tzoumodel_simp} if $\tT=0$, where we neglect the higher order terms and replace the time derivative by its fractional variant of Caputo type with order $\alpha \in(0,1)$. Hence, the adapted form of \eqref{eq:tzoumodel_simp} reads as
	\begin{equation}
		\label{eq:fracFourier}
		\left(1+\tq^\alpha \mathcal{D}_t^\alpha\right) \mathbf{q}(\X,t) = -\matrix{k}(\X)\nabla u (\X,t),
	\end{equation}
	where the parameter $\tq^\alpha$ is introduced to maintain the dimensions. The Caputo derivative $\mathcal{D}_t^\gamma$ of order $\gamma\in(0,1)$  is defined by \cite{Caputo1967,AtanackovicPilipovic2014book}
	\begin{equation*} 
		\label{eq:caputo}
		\mathcal{D}_t^\gamma f(t)= \partial_t \int_0^t \frac{(t-s)^{-\gamma}}{\Gamma(1-\gamma)} \left( f(s) - f(0)\right) \ds, \quad t \in [0,T],
	\end{equation*} 
	where $\Gamma$ denotes the Gamma function. 
	The energy conservation law is given as in the classical case
	\begin{equation}
		\label{eq:energy_conservation}
		\rho c \pdt u(\X,t) + \diver \mathbf{q}(\X,t) = G(\X,t),
	\end{equation} 
	where $G$ involves the heat sources and sinks. 
	Following \cite{SlodickaMaes,SINGH20112316,KUMAR201749}, we can write
	\begin{equation}\label{eq:source}
		\left(1+\tq^\alpha \mathcal{D}_t^\alpha\right) G(\X,t) = F(\X,t) - a\left(1+\tq^\alpha \mathcal{D}_t^\alpha\right) u(\X,t),
	\end{equation}
	where $a\geq 0$ is a constant and $F$ is a function representing the sources. Applying the operator $\mathbb{I}+\tq^\alpha \mathcal{D}_t^\alpha$ on \eqref{eq:energy_conservation}, and employing \eqref{eq:fracFourier} and \eqref{eq:source}, we obtain the following equation
	\begin{equation}
		\label{eq:1}
		\rho c \tq^\alpha \mathcal{D}_t^\alpha \pdt u (\X,t)  + a \tq^\alpha \mathcal{D}_t^\alpha u(\X,t)  +  \rho c \pdt u(\X,t) + \mathfrak{L}u(\X,t) = F(\X,t),
	\end{equation}
	for $(\X,t) \in Q_T,$
	where the operator $\mathfrak{L}$ is the following second-order linear differential operator
	\begin{equation*}
		\label{eq:L}
		\mathfrak{L}u(\X,t) := au(\X,t) + \diver\left( - \matrix{k}(\X) \nabla u (\X,t)\right).
	\end{equation*}
	
	The main target of this paper is to study the existence and uniqueness questions for the single-phase lag 
	heat equation \eqref{eq:1} subjected to given initial data and homogeneous Dirichlet boundary conditions. This SPL problem is given by
	\begin{equation}
		\label{eq:problem}
		\left\{
		\begin{array}{rll}
			\rho c \tqa \mathcal{D}_t^\alpha \pdt u(\X,t) + a\tqa \mathcal{D}_t^\alpha u(\X,t)&&\\ + \rho c \pdt u(\X,t) + \mathfrak{L} u(\X,t) &= F(\X,t), & (\X,t) \in Q_T \\
			u(\X,0) &= U_0(\X), & \X \in \Omega \\
			\pdt u (\X,0) &= V_0(\X), & \X \in \Omega \\
			u(\X,t) & = 0, & (\X,t) \in \Sigma_T.
		\end{array}
		\right.
	\end{equation}
	
	\subsection{Literature overview, new aspects and outline} \label{subsec_addedval_outline}

	First, we provide an overview of the main results when considering classical derivatives. 
	In \cite{Quintanilla2006}, the stability of the corresponding problem was investigated. Models obtained by first-order approximation for $u$ and second-order approximation for $\bm{q}$ in \eqref{eq:tzoumodel_simp} yield a hyperbolic equation which is (exponentially) stable if $\tq < 2\tT,$ whereas second-order approximation for both $u$ and $\bm{q}$ yield stability if $\tT >(2-\sqrt{3})\tq$ or $\tT \leq (2-\sqrt{3})\tq$, and the (restrictive) condition \cite[Eq. (3.10)]{Quintanilla2006} hold. For the first-order approximation considered in this paper, their result shows no conditions on the phase-lag parameters for stability. The well-posedness of these problems on smoothly bounded domains and convex domains was studied in \cite{Liu2020,Quintanilla2007} with the aid of the semigroup theory. In the case of first-order approximation for both $u$ and $\bm{q}$, the well-posedness on a bounded domain was investigated in \cite{Wang2001} for different boundary conditions and extended to higher dimensions in \cite{Wang2002}. 
	
	Fractional calculus has attracted many researchers because of the nonlocal property of the fractional derivatives and their use in the modelling part of complex systems, see, e.g. \cite{Sun2018b} and references therein. Fractional wave and diffusion equations (with constant order fractional derivatives)  were studied in e.g. \cite{Luchko2010,Sakamoto2011,Luchko2012,Kian2017,Kubica2018,Otarola2019} by means of eigenfunction expansions and in e.g. \cite{VanBockstal2020, VanBockstal2020b,VanBockstal2021,VanBockstal2022} by means of the Rothe method. For well-posedness results about equations of the form $\mathcal{D}_t^\alpha u + \pdt u(\X,t) + Lu =f$, we refer to e.g. \cite{Peszynska1996,Slodicka1997}. Similar results related to multi-term fractional equations can be found in \cite{Luchko2011,Stojanovic2011,Liu2013,Ye2013,Zheng2016}, the fractional telegraph equation \cite{Chen2008,Zhao2012,Hosseini2014,Kumar2014,Zhou2021} and the recent works about variable-order operators \cite{Wang2019d,Zheng2020b,VanBockstal2022d}. To the best of our knowledge, there is no general well-posedness result known for the SPL problem \eqref{eq:problem} under low regularity conditions. The goal of this paper is to fill in this gap. Note that the inverse source problems considered in the work \cite{Maes2020} assume the well-posedness of problem \eqref{eq:problem}.

	This paper is organised as follows. In Section~\ref{sec_prelim}, we fix the notations and assumptions, and give some valuable properties of the convolution kernel together with some technical results. The Fourier method will be used to obtain a representation formula for the solution in 1D in Section~\ref{subsec_sepvar}. Section~\ref{sec_exis} covers the existence of the solution. It includes the setup of the weak formulation and the time discretisation, as well as the a priori estimates and the Rothe functions. The uniqueness part is dealt with in Subsection~\ref{sec_unique}.
	
	
	\section{Preliminaries} \label{sec_prelim}
	\subsection{Notations and assumptions} \label{subsec_nota}
	
	We start by introducing the function spaces used throughout this work. 
	The classical $\Leb^2$-inner product is denoted by $\scal{\cdot}{\cdot}$  and the corresponding norm by $\nrm{\cdot }.$ Let $(X,\nrm{\cdot}_X)$ be a Banach space. Its dual space is denoted by $X^\ast.$  The duality pairing between ${\hko{1}}^\ast$ and $\hko{1}$ is denoted by $\inpd{\cdot}{\cdot}$ and is seen as a continuous extension of $\scal{\cdot}{\cdot}.$
	For $p\geq 1$ the space $\lpkIX{p}{X}$ is the space of all measurable functions $u \colon \Iopen \to X$ such that 
	\[
	\nrm{u}_{\lpkIX{p}{X}}^p := \int_0^T \nrm{u(t)}_X^p \dt < \infty.
	\]
	The space $\lpkIX{\infty}{X}$ is the space of all measurable functions $u\colon \Iopen \to X$ that are essentially bounded, i.e.
	\[
	\nrm{u}_{\lpkIX{\infty}{X}} := \esssup\limits_{t \in \Iclosed} \nrm{u(t)}_X < \infty.
	\]
	The space $\cIX{X}$ consists of all continuous functions $u\colon \Iclosed \to X$ such that
	\[
	\nrm{u}_{\cIX{X}} := \max_{t \in \Iclosed} \nrm{u(t)}_X < \infty.
	\]
	The space $\Hi^k\left((0,T),X\right)$ consists of all functions $u:\Iclosed \rightarrow X$ such that the weak derivative with respect to $t$ up to order $k$ exists and 
	\[
	\nrm{u}_{\Hi^k\left((0,T),X\right)}^2 := \int_0^T \left(\sum_{i=0}^k \nrm{u^{(i)}(t)}^2_X \right) \dt  < \infty.
	\]
	Similar notations will be used for the case of vector-valued functions. 
	
	Additionally, 
	the values of $C, \varepsilon$ and $C_\varepsilon$ are considered generic and positive constants. Their value can differ from place to place, but their meaning should be clear from the context. These constants are independent of the time discretisation parameter, where $\varepsilon$ is arbitrarily small and $C_\varepsilon$ arbitrarily large, i.e.\ $C_\varepsilon = C\left(1+\varepsilon + \frac{1}{\varepsilon}\right).$
	
	Now, we review some properties related to the operator $\mathfrak{L}$. The bilinear form $\mathcal{L}$ associated to the operator $\mathfrak{L}$ is given by
	\[
	\mathcal{L}\scal{u(t)}{\varphi} := \scal{\matrix{k}\nabla u(t)}{\nabla \varphi} + a\scal{u(t)}{\varphi}, \quad u(t),\varphi \in \hko{1}.
	\]
	for $u, \varphi \in \hko{1}.$ Here $a\geq 0$ and the matrix $\matrix{k}$ is a symmetric matrix-valued function on $\domain$ consisting of essentially bounded functions, i.e.
	\[
	\matrix{k} = \left(k_{i,j}(\X)\right) \in \Lp{\infty}:=\left(\Leb^{\infty}({\domain})\right)^{d \times d}, \quad \matrix{k} = \matrix{k}^\top.
	\]
	We assume $\matrix{k}$ to be uniformly elliptic, that is there exists a constant $\widetilde{k}>0$ such that for all $\X \in \domain$ and all $\bm{\xi} = \left(\xi_1,\dots,\xi_d\right)^\top  \in \RR^d$ it holds that
	\[
	\bm{\xi}^\top \cdot \matrix{k}(\X) \bm{\xi} = \sum_{i,j=1}^{d} k_{i,j}(\X) \xi_i \xi_j \geq \widetilde{k} \nrm{\bm{\xi}}^2.
	\] 
	These assumptions yield the following type of inequalities for $\mathcal{L}$:
	\begin{align}
		\mathcal{L}\scal{u}{\varphi} &\leq C\nrm{u}_{\hk{1}}\nrm{\varphi}_{\hk{1}}\nonumber \\
		\label{eq:LLell}
		\mathcal{L}\scal{\varphi}{\varphi} & \geq  \widetilde{k} \nrm{\nabla \varphi}^2,
	\end{align}
	for $u,\varphi \in \hko{1}.$  It follows that $\mathcal{L}$ is $\hko{1}$-elliptic by applying the Friedrichs' inequality on \eqref{eq:LLell}.
	
	Finally, we remember that  the fractional order appearing in \eqref{eq:problem} satisfies $0<\alpha <1$ and that the fractional phase-lag parameters $ \tq^\alpha$ is considered to be positive.  
	
	\subsection{Properties of the kernel} \label{subsec_g_prop} 
	
	We denote by $\ggamma$ the Riemann-Liouville kernel
	\[
	\ggamma(t) = \frac{t^{-\gamma}}{\Gamma(1-\gamma)}, \quad t>0, \gamma \in(0,1).
	\]
	The Caputo derivative operator $\mathcal{D}_t^\gamma$ can be rewritten as a convolution
	\[
	\mathcal{D}_t^\gamma f(t) = \pdt\left(\ggamma \ast (f-f(0))\right)(t),
	\]
	where $\ast$ denotes the Laplace convolution of two functions
	\[
	(k\ast z)(t) = \int_0^t k(t-s)z(s)\ds.
	\]
	The ensuing lemma contains several properties of the kernel $\ggamma,$ see e.g. \cite[Section~2]{VanBockstal2021}, \cite[Corollary 3.1]{VanBockstal2021} and \cite[Corollary~2]{Kubica2018}.
	\begin{lemma} \label{lemma:g}
		The function $\ggamma(t), t>0, 0<\gamma<1$ satisfies
		\begin{enumerate}[(i)]
			\item $\ggamma \in \Leb^1(0,T)$; 
			\item $\ggamma$ is decreasing in time $t$ and 
			$$
			\ggamma(t) \geq \frac{\min\{1,T^{-\gamma}\}}{\Gamma(1-\gamma)} >0, \quad t\in (0,T];
			$$
			\item $\pdt \ggamma \in \Leb^1_{\text{loc}}\left(0,T\right)$;
			\item $\ggamma(t) \geq 0, \pdt \ggamma(t) \leq 0, \pdtt \ggamma(t) \geq 0$ for all $t > 0$ and $\pdt \ggamma(t) \not\equiv 0,$ thus $\ggamma$ is strongly positive definite, i.e. for all $v \in \Leb^2_{\textit{loc}}\left( (0,\infty), \lp{2}\right)$ it holds that
			\[
			\int_0^t \scal{\left(\ggamma\ast v\right)(s)}{v(s)}\ds \geq 0, \quad t \geq 0;
			\]
			\item     For any $v \colon [0,T] \to \lp{2}$ satisfying $v \in \lpIlp$ with $\ggamma \ast v \in \Hi^1\left((0,T),\lp{2}\right),$ it holds for all $t \in [0,T]$ that
			\begin{equation}
				\label{eq:zacher_corol}
				\int_0^t \scal{\pdt \left(\ggamma\ast v \right)(s)}{v(s)}\ds \geq \frac{g_\gamma(T)}{2} \int_0^t \nrm{v(s)}^2_{\lp{2}}\ds.
			\end{equation}
		\end{enumerate}
	\end{lemma}
	
	Fractional integration 
	\[
	I^\alpha f(t) = \frac{1}{\Gamma(\alpha)}\int_0^t (t-s)^{\alpha-1} f(s)\ds = \left(g_{1-\alpha} \ast f\right)(t), \quad \text{for } \alpha > 0, f\in \Leb^1\Iopen, 
	\]
	has the property \cite[Eq. (2.21)]{Samko1993}
	\begin{equation}\label{eq:prop_frac_int}
		\left(I^{\alpha_1} (I^{\alpha_2} f)\right)(t) = \left(I^{\alpha_1+\alpha_2}f\right)(t), \quad \forall t\in\I,
	\end{equation}
	if $\alpha_1> 0, \alpha_2> 0$ and  $f \in \Cont([0,T]).$
	This property holds true for a.e. $t\in \I$ if $f\in \Leb^1\Iopen$ and it leads to the following consequence. 
	
	\begin{lemma}
		\label{lem:pdtconv}
		Let $u$ be absolutely continuous on $\I$ and assume that $\gamma\in(0,1)$. Then for a.a. $t\in\I$ it holds that
		\[
		\pdt (\ggamma \ast u) (t) = \ggamma(t) u(0) + (\ggamma \ast \pdt u)(t).
		\]
	\end{lemma}
	\begin{proof}
		We follow the lines of \cite[Lemma~2.1]{Kubica2017}.
		Since $u$ is absolutely continuous (then $\pdt u \in \Leb^1 \Iopen$ exists and $u(t)=u(a)+\int_a^t \pdt u(s)\ds$), we may write that 
		\begin{align*}
			\left(\ggamma\ast u\right)(t) &= \left(\ggamma\ast u(0)\right)(t) + \left(g_{\beta}\ast \int_0^{\bullet} \pdt u (s)\ds\right)(t)\\
			&= \left(\int_0^t \ggamma(s)\ds\right)u(0) + \int_0^t \ggamma(s)\int_0^{t-s}\pdt u(\tau)\dtau\ds,
		\end{align*}
		where the notation $\bullet$ is a placeholder for the independent variable.
		As $\partial_s \int_0^{t-s}\pdt u(\tau)\dtau = - \partial_{t-s} \int_0^{t-s}\pdt u(\tau)\dtau =- \pdt u(t-s)$, we get via partial integration that 
		\[
		\left(\ggamma\ast u\right)(t) = \left(\int_0^t \ggamma(s)\ds\right)u(0)   + \int_0^t \left(\int_0^s \ggamma(\tau) \dtau\right) \pdt u(t-s) \ds,
		\]
		since the boundary term vanishes. 
		Note that 
		\[
		\int_0^s \ggamma(\tau) \dtau = \int_0^s \frac{\tau^{-\gamma}}{\Gamma(1-\gamma)} \dtau = \frac{s^{1-\gamma}}{\Gamma(2-\gamma)} = \ggammamin(s). 
		\]
		Therefore, 
		\begin{equation}\label{thm1:eq1}
			\left(\ggamma\ast u\right)(t) = 
			\left(\int_0^t \ggamma(s)\ds\right)u(0)   + (\ggammamin\ast \pdt u)(t). 
		\end{equation}
		From \eqref{eq:prop_frac_int}, it follows that 
		\[
		(\ggammamin\ast \pdt u)(t)  = \left(I^{2-\gamma} \pdt u\right)(t)  = \left(I (I^{1-\gamma} \pdt u)\right)(t) = \left( I (\ggamma\ast \pdt u)\right)(t).
		\]
		Hence,
		\begin{align*} 
			\pdt  (\ggammamin\ast \pdt u)(t) &= \pdt I (\ggamma\ast \pdt u)(t) = (\ggamma\ast \pdt u)(t).
		\end{align*}
		Therefore, the result follows from differentiating \eqref{thm1:eq1} with respect to $t$. 
	\end{proof}
	
	\section{Solution via Fourier method in 1D} \label{subsec_sepvar} 
	In this part, we consider \eqref{eq:problem} with $k(x)=\bar k>0$ on the domain $\Omega =(0,L), L>0,$  and we suppose that the solution $u$ can be written as $u(x,t) =X(x)T(t).$ Employing the separation of variables technique leads to the eigenvalue problem
	\begin{equation*}
		\label{eq:problem_sep_var_space}
		\left\{
		\begin{array}{rll}
		-\bar k X''(x) + a X(x) &= \sigma  X(x), & \quad x \in(0,L)\\
			X(0) = X(L) &= 0,&
		\end{array}
		\right.
	\end{equation*}
	where $\sigma$ is the separation constant and $a>0$. The normalised solutions to this problem are given by
	\begin{equation*} \label{eq:Xn}
		X_n(x) = \sqrt{\frac{2}{L}} \sin\left(\sqrt{\frac{\sigma_n -a}{\bar k} } x\right) = \sqrt{\frac{2}{L}} \sin\left(\frac{n\pi}{L}x\right),
	\end{equation*} 
	where $\sigma_n = a +  \bar k\left(\frac{n\pi}{L}\right)^2$ for $n \in \mathbb{N}.$ 
	The functions $T_n(t)$ satisfy a fractional differential equation of the form
	\[
	\rho c \tau_q^\alpha \mathcal{D}_t^\alpha \pdt T + \rho c \pdt T + a\tau_q^\alpha \mathcal{D}_t^\alpha T + \sigma_n T = 0.
	\] 
	By means of the Laplace transform, we find that
	\[
	\mathcal{L}\left[ T_n(t)\right](s) = T_n(0) \left(\rho c \tau_q^\alpha s^\alpha + a \tau_q^\alpha s^{\alpha-1} + \rho c\right) g(s) + T'_n(0)\rho c \tau_q^\alpha s^{\alpha - 1} g(s),
	\] 
	where by \cite[Eqs. (5.31)-(5.33)]{Podlubny1998} the function $g(s)$ is given by
	\[
	g(s) = \frac{1}{\rho c \tau_q^\alpha} \sum_{m=0}^{\infty} (-1)^m \left(\frac{\sigma_n}{\rho c \tau^\alpha_q}\right)^m \sum_{k=0}^m \binom{m}{k} \left(\frac{a \tau_q^\alpha}{\sigma_n}\right)^k  \frac{s^{\alpha k - m - 1}}{\left(s^\alpha + \frac{1}{\tau_q^\alpha}\right)^{m+1}},
	\] 
	and forms a Laplace pair with
	\begin{multline}
		\label{eq:G(t)}
		G(t) = \frac{1}{\rho c \tau_q^\alpha } \sum_{m=0}^\infty \frac{1}{m!} \left(\frac{-\sigma_n}{\rho c \tau_q^\alpha}\right)^m \sum_{k = 0}^m \binom{m}{k} \left(\frac{a \tau_q^\alpha}{\sigma_n}\right)^k \\ \quad \times  t^{(\alpha +1)(m+1) - \alpha k - 1} E^{(m)}_{\alpha, \alpha + 1 + m - \alpha k} \left(-\frac{1}{\tau_q^\alpha} t^\alpha \right).
	\end{multline}
	Here $E_{\gamma, \delta}(z) = \sum_{k=0}^\infty \frac{z^k}{\Gamma(\beta + \alpha k)}$ is the two parameter Mittag-Leffler function, see \cite[Eq. (1.56)]{Podlubny1998}. 
	Using properties of the Laplace transform related to fractional derivatives \cite[§2.8]{Podlubny1998}, we find that
	\begin{equation}\label{eq:laplace_prop}
		\mathcal{L}\left[\mathcal{D}_t^\alpha G(t) \right](s) = s^\alpha g(s) \quand \mathcal{L}\left[ \int_0^t \mathcal{D}_t^\alpha G(z)\mathrm{d}z \right](s) = s^{\alpha - 1} g(s),
	\end{equation}
	noting that $G(0) = 0.$ 
	
	The multinomial Mittag-Leffler function, see \cite[Eq. (3.16)]{Luchko1996}, is given by
	\begin{equation} \label{eq:MML}
		E_{(\alpha_1, \dots, \alpha_m),\beta}(z_1, \dots, z_m) = \sum_{k=0}^\infty \sum_{\substack{ k_1+\dots + k_m = k\\ k_j\geq 0}} \binom{k}{k_1,\dots, k_m} \frac{\prod_{j=1}^m z_j^{k_j}}{ \Gamma(\beta + \sum_{j=1}^m \alpha_j k_j)},
	\end{equation} 
	where $\binom{k}{k_1,\dots,k_m}$ is the multinomial coefficient. Let us remark that \eqref{eq:MML} is invariant under (the same) permutation on the parameters $\alpha_j$ and the variables $z_j$ with $j=1,\dots,m.$ 
	
	Following the approach from \cite{LiLi2020}, we can rewrite \eqref{eq:G(t)} by means of multinomial Mittag-Leffler functions. After relabelling, we obtain
\begin{equation}
    	\label{eq:Gt}
		G(t) = \frac{t^\alpha}{\rho c \tau_q^\alpha} E_{(\alpha+1,1,\alpha), \alpha +1}\left(\frac{-\sigma_n}{\rho c \tau_q^\alpha} t^{\alpha+1}, \frac{-a}{\rho c} t,\frac{-1}{\tau_q^\alpha}t^\alpha\right).
\end{equation}

	We now list some useful lemmas, the first one is a slight generalisation of \cite[Lemma~3.3]{LiLiuYamamoto2015}, which we recover by setting $\gamma = \beta.$
	\begin{lemma}
		\label{lem:diffMML}
		Let $\alpha_1,\dots, \alpha_m$ be positive constants and $q_1, \dots, q_m, \beta$ real. Then, for real $\gamma$, we have that
		\begin{align*} 
			&\frac{\mathrm{d}}{\mathrm{d}t} \left[ t^\gamma E_{(\alpha_1, \dots, \alpha_m), \beta +1}\left(q_1 t^{\alpha_1}, \dots, q_m t^{\alpha_m}\right)\right]\\   &\quad = t^{\gamma-1} E_{(\alpha_1, \dots, \alpha_m),\beta}\left(q_1 t^{\alpha_1}, \dots, q_m t^{\alpha_m}\right) \\   & \qquad+(\gamma-\beta) t^{\gamma-1} E_{(\alpha_1,\dots,\alpha_m), \beta +1}\left(q_1 t^{\alpha_1}, \dots, q_m t^{\alpha_m}\right). 
		\end{align*}
	\end{lemma}
	\begin{proof}
		First, we calculate
		\begin{align*}
			& \frac{\mathrm{d}}{\mathrm{d}t} \left[t^\beta E_{(\alpha_1, \dots, \alpha_m),\beta +1}\left(q_1 t^{\alpha_1}, \dots, q_m t^{\alpha_m}\right)\right] \\
			&\quad =  \frac{\mathrm{d}}{\mathrm{d}t} \left[ \sum_{k=0}^\infty \sum_{k_1+\dots + k_m=k} \binom{k}{k_1,\dots, k_m} \prod_{j=1}^m q_j^{k_j} \frac{t^{\beta + \sum_{j=1}^m \alpha_j k_j}}{\Gamma\left(\beta + 1+ \sum_{j=1}^m \alpha_j k_j\right)} \right] \\
			&\quad = \sum_{k=0}^\infty \sum_{k_1+ \dots + k_m = k} \binom{k}{k_1,\dots, k_m} \prod_{j=1}^m q_j^{k_j} \frac{t^{\beta - 1 + \sum_{j=1}^m \alpha_j k_j}}{\Gamma\left(\beta + \sum_{j=1}^m \alpha_j k_j\right)} \\
			&\quad =  t^{\beta-1} E_{(\alpha_1,\dots, \alpha_m), \beta}\left(q_1 t^{\alpha_1}, \dots, q_m t^{\alpha_m}\right).
		\end{align*}
		Then, an application of the chain rule  shows that
		\begin{align*}
			& \frac{\mathrm{d}}{\mathrm{d}t} \left[t^\gamma E_{(\alpha_1,\dots, \alpha_m), \beta +1}\left(q_1 t^{\alpha_1}, \dots, q_m t^{\alpha_m}\right) \right] \\
			&\quad =  \frac{\mathrm{d}}{\mathrm{d}t} \left[ t^{\gamma - \beta} t^{\beta} E_{(\alpha_1,\dots, \alpha_m),\beta+1}\left(q_1 t^{\alpha_1}, \dots, q_m t^{\alpha_m}\right) \right] \\
			&\quad = t^{\gamma - \beta} t^{\beta -1} E_{(\alpha_1, \dots, \alpha_m), \beta}\left(q_1 t^{\alpha_1}, \dots, q_m t^{\alpha_m}\right)\\
			&\qquad + (\gamma - \beta) t^{\gamma-\beta -1} t^{\beta} E_{(\alpha_1, \dots, \alpha_m), \beta+1}\left(q_1 t^{\alpha_1}, \dots, q_m t^{\alpha_m}\right),
		\end{align*}
		which proves the result.
	\end{proof}
	
	The next result is a generalization of the property $E_{\alpha, 0}(z) = z E_{\alpha, \alpha}(z)$ for the two-parameter Mittag-Leffler function. 
	\begin{lemma} \label{lem:MMLsum}
		Let $\alpha_1, \dots, \alpha_m$ be positive constants, $z_1,\dots,z_m \in \mathbb{C}$ and $\beta>0$ be fixed. Then
		\[
		\sum_{j=1}^m z_j E_{(\alpha_1, \dots, \alpha_m), \beta + \alpha_j}\left(z_1, \dots, z_m\right) + \frac{1}{\Gamma(\beta)} = E_{(\alpha_1, \dots, \alpha_m),\beta}\left(z_1, \dots, z_m\right).
		\]
	\end{lemma}
	The proof can be performed as in \cite[Lemma~3.1]{LiLiuYamamoto2015}. Let us remark that this result remains valid in the limit case $\beta\to 0,$ as $1/\Gamma(\beta) \to 0.$ Next, we extend \cite[Lemma~3.2]{LiLiuYamamoto2015} such that this lemmata can be used in our problem setting (multinomial Mittag-Leffler function \eqref{eq:MML} with $\alpha_1 \in [1,2)$) as well.

	\begin{lemma}\label{lem:MMLbound}
		Let $\beta >0$ and $\alpha_1>\alpha_2 > \dots > \alpha_m>0$ be given with $\alpha_1 \in (0,2)$. Assume that $\alpha_1 \pi /2 < \mu < \min\{ \alpha_1\pi, \pi\}, \mu \leq \abs{\argument(z_1)} \leq \pi$ and there exists $K>0$ such that $-K \leq z_j<0$ for $j=2,\dots,m.$ Then there exists a constant $C>0$ depending only on $\mu, K, \alpha_j, j=1,\dots,m$ and $\beta$ such that
		\[
		\abs{E_{(\alpha_1, \dots,\alpha_m), \beta}\left(z_1,\dots,z_m\right)} \leq \frac{C}{1+ \abs{z_1}}.
		\]
	\end{lemma}
	\begin{proof}
		The proof goes along the same lines as that of \cite[Lemma 3.2]{LiLiuYamamoto2015}, therefore we point out some key steps and comment on the differences. The starting point is the contour representation of the reciprocal of the gamma function, \cite[Eq. (1.52)]{Podlubny1998}
		\begin{equation} \label{eq:gamma_recipr}
			\frac{1}{\Gamma(z)} = \frac{1}{2\pi \alpha_1 i} \int_{\gamma(R,\theta)} \exp\left(\zeta^{1/\alpha_1}\right) \zeta^{(1-z-\alpha_1)/\alpha_1} \mathrm{d}\zeta,
		\end{equation} 
		which is applicable since $\alpha_1 < 2$ and $\pi\alpha_1 /2 < \mu < \min\{ \alpha_1\pi, \pi\}.$ Here $\gamma(R,\theta)$ is the contour consisting of an circular arc $\{\zeta \in \CC  : \abs{\zeta} =R, \abs{\argument(\zeta)} \leq \theta\} $ and two half lines $\{ \zeta \in \CC: \abs{\zeta}>R, \abs{\argument(\zeta)} =  \pm \theta \}. $ The radius $R$ is to be fixed later and $\theta$ is chosen such that $ \alpha_1 \pi / 2 < \theta < \mu.$ 
		By means of \eqref{eq:gamma_recipr}, we can rewrite
		\begin{align*}
			&E_{(\alpha_1,\dots, \alpha_m),\beta}(z_1,\dots,z_m) \\
			&\quad =  \frac{1}{2\pi \alpha_1 i } \int_{\gamma(R,\theta)} \exp\left(\zeta^{1/\alpha_1}\right) \zeta^{\frac{1-\beta}{\alpha_1} - 1} \sum_{k = 0}^\infty \sum_{k_1 + \dots + k_m = k} \binom{k}{k_1, \dots, k_m} \\ &\qquad \qquad \qquad \qquad\times\prod_{j=1}^m z_j^{k_j} \zeta^{- \sum_{j=1}^m \frac{\alpha_j}{\alpha_1}k_j} \mathrm{d} \zeta\\
			&\quad = \frac{1}{2\pi \alpha_1 i } \int_{\gamma(R,\theta)} \exp\left(\zeta^{1/\alpha_1}\right) \zeta^{ \frac{1-\beta}{\alpha_1} - 1} \sum_{k=0}^{\infty} \left( \frac{z_1}{\zeta} + \sum_{j=2}^m z_j \zeta^{-\alpha_j/\alpha_1}\right)^k  \mathrm{d}\zeta,
		\end{align*}
		by isolating the index $j=1$ and applying the multinomial theorem. Note that $1- \frac{\alpha_j}{\alpha_1} \in (0,1)$ for $j =2,\dots,m.$ Therefore, taking $R >  \abs{z_1} + K \sum_{j=2}^m R^{1-\frac{\alpha_j}{\alpha_1}}$ ensures convergence of the series. In case all $\abs{z_j}\leq K, j = 1,\dots,m$ we can fix $R$ as a constant depending only on $K$ and $\alpha_1,\dots, \alpha_m.$ In that case we deduce that 
		\[
		E_{(\alpha_1,\dots, \alpha_m), \beta} (z_1,\dots,z_m) = \frac{1}{2\pi \alpha_1 i} \int_{\gamma(R,\theta)} \frac{\exp\left(\zeta^{1/\alpha_1}\right) \zeta^{\frac{1-\beta}{\alpha_1}}} {\zeta - z_1 - \sum_{j=2}^m z_j \zeta^{1-\frac{\alpha_j}{\alpha_1}}}\mathrm{d}\zeta.
		\] 
		Repeating the arguments of \cite{LiLiuYamamoto2015} we obtain in case $\abs{z_1}>R$ and $\mu \leq \abs{\argument(z_1)} \leq \pi$ that 
		\[
		\abs{E_{(\alpha_1, \dots, \alpha_m), \beta}(z_1,\dots,z_m)} \leq \frac{C}{\abs{z_1}},
		\] 
		and in case $\abs{z_1} \leq R, \mu \leq \abs{\argument(z_1)} \leq \pi$ that 
		\[
		\abs{E_{(\alpha_1,\dots,\alpha_m),\beta}(z_1,\dots,z_m)} \leq C E_{\alpha_m ,\beta} \left(R + (m-1)K\right),
		\] 
		where $C \geq \max\{ 1, \frac{\beta+\alpha_1}{\alpha_m}\}.$ To proof this last inequality, we used the bound
		\begin{equation} \label{eq:gamma_bound}
			\frac{1}{\Gamma(\beta + \sum_{j=1}^m \alpha_j k_j)} \leq \frac{\max\{1, \frac{\beta + \alpha_1}{\alpha_m}\}}{\Gamma(\beta + \alpha_mk)}.
		\end{equation} 
		This last result \eqref{eq:gamma_bound} can be shown in the following way. For $k =0$ there is nothing to show, so assume $k\geq 1.$ For $\beta \geq 1,$ we have $\Gamma\left(\beta + \sum_{j=1}^m \alpha_j k_j\right) > \Gamma\left(\beta + \alpha_m k\right), $ since $\Gamma(z) $ is monotonically increasing for $z\geq 1$ and $\alpha_1>\alpha_2>\dots>\alpha_m>0.$ For $0< \beta < 1$, we use that $\Gamma(z+1) = z\Gamma(z)$ to write
		\begin{align*}
			\frac{\Gamma\left(\beta + \alpha_m k \right)}{\Gamma\left(\beta + \sum_{j = 1}^m \alpha_j k_j\right)} &= \frac{\Gamma\left(\beta + 1 + \alpha_m k \right) \left(\beta + \sum_{j=1}^m \alpha_jk_j\right)}{\Gamma\left(\beta + 1 + \sum_{j=1}^m \alpha_j k_j\right)\left(\beta + \alpha_m k\right)} \\
			&\leq \frac{\beta + \sum_{j=1}^m \alpha_j k_j}{\beta + \alpha_m k } \leq \frac{\beta}{\alpha_mk} + \frac{\alpha_1}{\alpha_m} \\
			&\leq \frac{\beta + \alpha_1}{\alpha_m},
		\end{align*}
		which is a constant that can be brought out of the summations.
	\end{proof}
	
	Using Lemma~\ref{lem:MMLsum} or scrutinising the above proof, we see that Lemma~\ref{lem:MMLbound} also stays valid in the limit case $\beta\to 0.$ Now, we continue solving the one-dimensional problem considered in this section. Employing \eqref{eq:laplace_prop}, the solution $T_n(t)$ takes the form
	\[
	T_n(t) = c_n T_n^1(t) + d_n T_n^2(t), \quad c_n,d_n\in\RR,
	\]
	with 
	\begin{align*}
		T_n^1(t) &= \rho c\tau_q^\alpha \mathcal{D}_t^\alpha G(t) + a \tau_q^\alpha \int_0^t \mathcal{D}_t^\alpha G(s)\mathrm{d}s + \rho c G(t) \\
		T_n^2(t) &= \rho c \tau_q^\alpha \int_0^t \mathcal{D}_t^\alpha G(s)\mathrm{d}s,
	\end{align*}
	where $G(t)$ is given by \eqref{eq:Gt}. The results in \cite[Theorem 2.3]{Bazhlekova2021} show that
	\begin{align*}
	    \mathcal{D}_t^\alpha G(t) &= \frac{1}{\rho c \tau_q^\alpha} E_{(\alpha+1,1,\alpha),1} \left( \frac{-\sigma_n}{\rho c \tau_q^\alpha} t^{\alpha+1}, \frac{-a}{\rho c}t,\frac{-1}{\tau_q^\alpha}t^\alpha\right) \\
 		\int_0^t\mathcal{D}_t^\alpha G(s)\mathrm{d}s &= \frac{t}{\rho c \tau_q^\alpha} E_{(\alpha+1,1,\alpha),2} \left( \frac{-\sigma_n}{\rho c \tau_q^\alpha} t^{\alpha+1},\frac{-a}{\rho c}t, \frac{-1}{\tau_q^\alpha}t^{\alpha}\right).
	\end{align*}
	Hence, we find that
	\begin{multline*}
		\label{eq:Tn}
		T_n^1(t) =  E_{(\alpha+1,1,\alpha),1}\left( \frac{-\sigma_n}{\rho c \tau_q^\alpha} t^{\alpha+1},\frac{-a}{\rho c}t, \frac{-1}{\tau_q^\alpha}t^{\alpha}\right) \\
		+ \frac{a}{\rho c} tE_{(\alpha+1,1,\alpha),2}\left( \frac{-\sigma_n}{\rho c \tau_q^\alpha} t^{\alpha+1},\frac{-a}{\rho c}t, \frac{-1}{\tau_q^\alpha}t^{\alpha}\right) \\
		+ \frac{t^\alpha}{\tau_q^\alpha} E_{(\alpha+1,1,\alpha),\alpha+1}\left( \frac{-\sigma_n}{\rho c \tau_q^\alpha} t^{\alpha+1},\frac{-a}{\rho c}t, \frac{-1}{\tau_q^\alpha}t^{\alpha}\right)  \nonumber
	\end{multline*}
	and
\[
	T_n^2(t) =  t E_{(\alpha+1,1,\alpha),2}\left( \frac{-\sigma_n}{\rho c \tau_q^\alpha} t^{\alpha+1}, \frac{-a}{\rho c}t, \frac{-1}{\tau_q^\alpha}t^{\alpha}\right).
	\]
	From Lemma~\ref{lem:MMLsum}, we have that
	\begin{multline*}
		E_{(\alpha+1,1,\alpha),1}\left( \frac{-\sigma_n}{\rho c \tau_q^\alpha} t^{\alpha+1},\frac{-a}{\rho c}t, \frac{-1}{\tau_q^\alpha}t^{\alpha}\right)\\
		= 1 - \frac{\sigma_n}{\rho c \tau_q^\alpha} t^{\alpha+1} E_{(\alpha+1,1,\alpha),\alpha+2}\left( \frac{-\sigma_n}{\rho c \tau_q^\alpha} t^{\alpha+1}, \frac{-a}{\rho c}t, \frac{-1}{\tau_q^\alpha}t^{\alpha}\right) \\
		-\frac{a}{\rho c}t E_{(\alpha+1,1,\alpha),2}\left( \frac{-\sigma_n}{\rho c \tau_q^\alpha} t^{\alpha+1},\frac{-a}{\rho c}t, \frac{-1}{\tau_q^\alpha}t^{\alpha}, \right).\\
		-\frac{1}{\tau_q^\alpha}t^{\alpha} E_{(\alpha+1,1,\alpha),\alpha+1}\left( \frac{-\sigma_n}{\rho c \tau_q^\alpha} t^{\alpha+1},\frac{-a}{\rho c}t, \frac{-1}{\tau_q^\alpha}t^{\alpha} \right)
	\end{multline*}
	Hence, $T_n^1$ simplifies to 
	\[
	T_n^1(t) = 1 - \frac{\sigma_n}{\rho c \tau_q^\alpha} t^{\alpha+1} E_{(\alpha+1,1,\alpha),\alpha+2}\left( \frac{-\sigma_n}{\rho c \tau_q^\alpha} t^{\alpha+1},\frac{-a}{\rho c} t ,\frac{-1}{\tau_q^\alpha}t^{\alpha} \right). 
	\]
	Moreover, from Lemma~\ref{lem:MMLbound}, it follows that
	\[
	\max_{t\in \I}\abs{T_n^i(t)} \le T_i, \quad \forall n\in\NN, \quad i=1,2. 
	\]
	The unknown function $u(x,t)$ can be represented as the series
	\[
	u(x,t) = \sum_{n=1}^{\infty} X_n(x)\left[(U_0,X_n)T_n^1(t) + (V_0,X_n) T_n^2(t)\right].
	\] 
	Hence, we have for all $t\in\I$ that 
	\begin{align*}
		\nrm{u(t)}_{\Leb^2(0,L)}^2 &=  \sum_{n=1}^{\infty} \abs{(U_0,X_n)T_n^1(t) + (V_0,X_n) T_n^2(t)}^2    \\
		& \le 2\max\{T_1^2,T_2^2\}  \sum_{n=1}^{\infty} \left[  \abs{(U_0,X_n)}^2 + \abs{(V_0,X_n)}^2 \right] \\
		& \le 2\max\{T_1^2,T_2^2\} \left(\nrm{U_0}_{\Leb^2(0,L)}^2 + \nrm{V_0}_{\Leb^2(0,L)}^2\right).
	\end{align*}
	Using Lemma~\ref{lem:diffMML}, we obtain that 
    	\begin{equation*} 
		\left(T^1_n\right)^\prime(t) =    \frac{-\sigma_n}{\rho c \tau_q^\alpha} t^\alpha E_{(\alpha+1,1, \alpha), \alpha+1}\left( \frac{-\sigma_n}{\rho c \tau_q^\alpha} t^{\alpha+1}, \frac{-a}{\rho c}t, \frac{-1}{\tau_q^\alpha}t^{\alpha}\right).
	\end{equation*}
	and
	\[
	\left(T^2_n\right)^\prime(t) =  E_{(\alpha+1,1,\alpha),1}\left( \frac{-\sigma_n}{\rho c \tau_q^\alpha} t^{\alpha+1}, \frac{-a}{\rho c}t, \frac{-1}{\tau_q^\alpha}t^{\alpha}\right).
	\]
Therefore, from Lemma~\ref{lem:MMLbound},
it follows for all $n\in\NN$ that
\[
\abs{\left(T^1_n\right)^\prime(t)} \le \left(\frac{\sigma_n}{\rho c \tau_q^\alpha}  \right)^{\half} \frac{\left(\frac{\sigma_n}{\rho c \tau_q^\alpha}  \right)^\half t^\alpha C}{1+\frac{\sigma_n}{\rho c \tau_q^\alpha} t^{\alpha+1}} \le \sigma_n^{\half} \widetilde{M} C =: \sigma_n^{\half}  \dot T_1 
\]
and
\[ 
\max_{t\in \I}\abs{\left(T^2_n\right)^\prime(t)} \le \dot T_2, 
\]
where
\[
    \widetilde{M}:=\max_{(r,t)\in [0,\infty]\times [0,T]} \frac{r^\half t^\alpha}{1+rt^{\alpha+1}} < +\infty. 
\]
Using these bounds, we have for all $t\in(0,T]$ that 
\begin{align*}
\nrm{\pdt u(t)}_{\Leb^2(0,L)}^2 &=  \sum_{n=1}^{\infty} \abs{(U_0,X_n)\left(T^1_n\right)^\prime(t) + (V_0,X_n) \left(T^2_n\right)^\prime(t)}^2    \\
& \le 2 \dot T_1^2 \sum_{n=1}^{\infty}  \sigma_n \abs{(U_0, X_n)}^2 + 2 \dot T_2^2 \sum_{n=1}^\infty\abs{(V_0,X_n)}^2 \\
&\le 2\max\{\dot T_1^2,\dot T_2^2\} \left( \nrm{U_0}_{\Hi^1_0(0,L)}^2+ \nrm{V_0}_{\Leb^2(0,L)}^2 \right).
\end{align*}
Hence, we can conclude that 
\[
\max_{t\in \I} \nrm{u(t)}_{\Leb^2(0,L)}^2 + \max_{t\in \I} \nrm{\pdt u(t)}_{\Leb^2(0,L)}^2  \le C
\]
if $U_0\in \Hi_0^1(0,L)$ and $V_0\in \Leb^2(0,L)$.

	
	\section{The single-phase-lag problem} \label{sec_exis} 
	
	This section is devoted to constructing a weak solution to \eqref{eq:problem} through Rothe's method \cite{Kavcur1985}. This section is structured as follows. First, the variational formulation and its discrete analogue are posed. Next, the existence and uniqueness of a solution $u_i$ to the discrete variational problem at each time slice $i$ is deduced. Given those solutions, we prove some a priori estimates in appropriate norms. Finally, we construct the Rothe functions and show that they possess a converging subsequence whose limit constitutes a weak solution of \eqref{eq:problem}. In the following section, the uniqueness of a weak solution will be established, implying that the whole Rothe sequence converges to this weak solution.
	
	\subsection{Weak formulation} \label{subsec_VF} 
	
	Note that the convolution map $f \mapsto\galpha  \ast f$ is a bounded operator $\lpIlp \to \lpIlp$ with norm bounded by $\nrm{\galpha}_{\Leb^1(0,T)},$ as can be seen from Young's inequality for convolutions. Therefore, if $\pdt u \in \lpIlp$, by Lemma~\ref{lem:pdtconv}, we have that $\galpha \ast \pdt u  = \pdt \left( \galpha\ast (u-U_0)\right)$ is an element of $\lpIlp.$
	Now, the weak formulation can be stated as follows.
	
	\begin{definition}[Weak formulation (SPL)] \label{def:wf}
		Find $u\in \lpkIX{\infty}{\hko{1}}\cap \cIX{\lp{2}}$ with 
		$\pdt u \in \lpkIX{2}{\lp{2}} $ and 
		$\pdt\left(\galpha\ast \left(\pdt u - V_0\right)\right) \in \lpkIX{2}{{\hko{1}}^\ast}$ such that for a.a.\ $t \in (0,T)$ it holds that
		\begin{multline}
			\label{eq:VF} 
			\rho c \tq^\alpha \inpd{\pdt\left(\galpha\ast\left(\pdt u - V_0\right)\right)(t)}{\varphi}+ a \tq^\alpha \scal{\left(\galpha\ast \pdt u \right)(t)}{\varphi} \\ + \rho c \scal{\pdt u(t)}{\varphi} + \mathcal{L}\scal{u(t)}{\varphi} = \scal{F(t)}{\varphi},
		\end{multline}
		for all $\varphi \in \hko{1}.$
	\end{definition}

	\subsection{Time discretisation} \label{subsec_time_disc}
	Let $n\in\NN$ be given. We discretise the time interval $\Iopen$ in $n$ subintervals according to the nodes $t_i = i\tau$ for $i=0,1,\dots,n$ where $\tau = T/n$ is the time step. We define $\lceil t\rceil_{\tau} = t_i$ for $t\in(t_{i-1},t_i]$.  For a function $z(t)$ we will write $z_i = z(t_i)$ for its evaluation at the time steps. The backward Euler method is used for approximating the first and second order derivatives at each $t_i$, i.e. 
	\begin{equation}
		\label{eq:backwEuler}
		\delta z_i = \frac{z_i - z_{i-1}}{\tau} \quand \delta^2 z_i = \frac{\delta z_i - \delta z_{i-1}}{\tau} = \frac{z_i - z_{i-1} }{\tau^2} - \frac{\delta z_{i-1}}{\tau}.
	\end{equation}
	We set $u_0 = U_0$ and $\delta u_0 = V_0$ according to the initial conditions.
	
	Given a kernel $\kappa\colon (0,T] \to \mathbb{R}$ and a function $z \colon [0,T]\to \mathbb{R}$, we define the following discrete convolution approximations  of $(\kappa \ast z)(t_i)$:
	\begin{align}
		\left(\kappa \ast z \right)^{c}_i &:= \sum_{\ell=1}^i \kappa_{i+1-\ell}z_\ell  \tau. \label{eq:disconv1}
	\end{align}
	If $z_0=0$, then we can properly define $\left(\kappa \ast z \right)^{c}_0:=0.$
	
	We record the following two lemmas which will be needed and crucial in the estimates below. It is a discrete version of the inequality \eqref{eq:zacher_corol}. For the proof, we refer to \cite[Lemma~3.2]{Slodicka2016}. 
	\begin{lemma}
		\label{lem:ex1}
		Let	$\tau = T/n$ be the time step, where $n\in \NN$ is the number of time discretisation intervals. Let $\{z_i\}_{i\in\NN\cup\{0\}}$ and $\{\kappa_i\}_{i\in\NN}$ be two sequences of real numbers. Assume that $\kappa_{i+1} \leq \kappa_{i}$ for all $i\in \NN$ and $(\kappa \ast z)^{c}_0 = 0$. Then
		\begin{equation*}
			\label{eq:Dconv}
			2 \delta (\kappa\ast z)^{c}_i z_i \geq \delta \left(\kappa \ast z^2\right)^{c}_i + \kappa_i z_i^2
		\end{equation*}
	\end{lemma} 
	As a consequence, we get the following lemma.
	\begin{lemma}
		\label{lem:lemconv}
		Let the assumptions of Lemma~\ref{lem:ex1} be fulfilled. Then, it holds that
		\begin{equation*}
			\label{eq:convrelation}
			2\sum_{i=1}^j	 \delta (\kappa\ast z)^{c}_i z_i \geq \left(\kappa \ast z^2\right)^{c}_j  + \sum_{i=1}^j \kappa_{i} z_i^2.
		\end{equation*}
		assuming $\left( \kappa \ast \nrm{z}^2\right)^{c}_0 =0$.
	\end{lemma}
	
	For convenience, we state the summation by parts formula for bilinear mappings.

	\begin{lemma}
		\label{lem:exPP}
		Let $b\colon V \times V \to \RR$ be a bilinear form on a vector space $V$ and let $\{z_i\}_{i\in\NN}, \{w_i\}_{i\in\NN}$ be two sequences in $V.$ Then
		\begin{equation*}
			\label{eq:pp}
			\sum_{i=1}^j b\left(z_i,w_i-w_{i-1}\right) 
			= b\left(z_j,w_j\right) - b\left(z_0,w_0\right) - \sum_{i=1}^j b\left( \delta z_i, w_{i-1}\right) \tau.
		\end{equation*}
	\end{lemma}
	
	\subsection{Existence at each time step} 
	The discrete variational formulation is obtained by approximating \eqref{eq:VF} at time $t =t_i$ by means of \eqref{eq:backwEuler} and \eqref{eq:disconv1}.
	\begin{definition}[Discrete weak formulation (SPL)] Find $u_i \in \hko{1}, i =1,\dots,n,$ such that     \begin{multline}
			\label{eq:VFI}
			\rho c\tq^\alpha \scal{\delta\left(\galpha \ast (\delta u-V_0)\right)^{c}_i}{\varphi} + a\tq^\alpha \scal{\left(\galpha \ast \delta u\right)^{c}_i}{\varphi} \\  + \rho c \scal{\delta u_i}{\varphi} + \mathcal{L}\scal{u_i}{\varphi} = \scal{F_i}{\varphi},
		\end{multline}
		for all $\varphi \in \hko{1}.$
	\end{definition}
	
	The existence of a solution on a single time step is established in the following lemma. 
	
	\begin{lemma}
		\label{lem:LMexun}
		Suppose that $F \in \Leb^2\left([0,T],\lp{2}\right)$, $U_0 \in \lp{2}$ and $V_0 \in \lp{2}.$ Then for any $i=1,\dots,n$ there exists a unique solution $u_i \in\hko{1}$  to \eqref{eq:VFI}.
	\end{lemma}
	\begin{proof}
		We have for $i\geq 1$ that 
		\[
		\delta\left(\galpha \ast (\delta u-V_0)\right)^{c}_i = \left(\galpha \ast \delta(\delta u-V_0)\right)^{c}_i = \left(\galpha \ast \delta^2 u\right)^{c}_i.
		\]
		Hence, using \eqref{eq:backwEuler} and \eqref{eq:disconv1}, the discrete problem \eqref{eq:VFI} can be rewritten into the form
		\begin{equation*}
			\label{eq:af}
			\mathcal{A}\scal{u_i}{\varphi} = \mathcal{F}_i(\varphi) 
			,\quad \text{ for all } \varphi \in \hko{1},
		\end{equation*} 
		where $\mathcal{A}$ is the  $\hko{1}$-elliptic and continuous bilinear form given by
		\begin{equation*}
			\label{eq:bila}
			\mathcal{A} \scal{u_i}{\varphi}:=
			\left(\frac{\rho c \tq^\alpha}{\tau} \galpha(\tau) + a \tq^\alpha \galpha(\tau) + \frac{\rho c}{\tau}\right) \scal{u_i}{\varphi} + \mathcal{L}\left(u_i,\varphi\right),
		\end{equation*}
		and
		\begin{multline*}
			\mathcal{F}_i(\varphi):= \scal{F_i}{\varphi} + \rho c \tq^\alpha \galpha(\tau) \scal{\frac{u_{i-1}}{\tau} +\delta u_{i-1}}{\varphi}+\frac{\rho c}{\tau} \scal{u_{i-1}}{\varphi} \\- \rho c \tq^\alpha \sum_{k=1}^{i-1}\galpha(t_{i+1-k}) \scal{\delta^2 u_k}{\varphi}\tau  + a\tq^\alpha \galpha(\tau) \scal{u_{i-1}}{\varphi}\\ - a\tq^\alpha \sum_{k=1}^{i-1} \galpha(t_{i+1-k}) \scal{\delta u_k}{\varphi}\tau  ,
		\end{multline*}
		where in case $i = 1$ the summations yield no contribution. The boundedness of $\mathcal{F}_i$ follows inductively under the assumptions 
		$F \in \lpkIX{2}{{\lp{2}}}$ and $U_0 \in \hko{1}, V_0 \in \lp{2}.$ An application of the Lax-Milgram lemma \cite[Theorem~18.E]{Zeidler1990IIA} consecutively gives the existence and uniqueness of $u_i \in \hko{1}, i =1,\dots,n$ to the discrete problem \eqref{eq:VFI}.
	\end{proof}

	\subsection{A priori estimates} \label{subsec_estimates}
	Our next goal is to derive estimates on the discrete solutions $u_i$ and
	various related objects. These estimates will be crucial in the convergence study in the following subsection.
	
	\begin{lemma}
		\label{lem:1}
		Suppose that $F \in \Leb^2\left([0,T],\lp{2}\right)$, $U_0 \in \hko{1}$ and $V_0 \in \lp{2}.$ Then, there exist a positive constant $C$  
		such that for every $j = 1,\dots,n$ 
		the following estimate holds
		\begin{equation*}
			\left( \galpha \ast \nrm{\delta u-V_0}^2\right)^{c}_j + \sum_{i =1}^j \nrm{\delta u_i}^2 \tau  + \nrm{u_j}_{\hk{1}}^2 + \sum_{i=1}^{j} \nrm{u_i - u_{i-1}}_{\hk{1}}^2  \leq C.
		\end{equation*}
	\end{lemma}
	\begin{proof}
		Take $\varphi = \delta u_i \tau$ in \eqref{eq:VFI} and sum the resulting equations for $i = 1,\dots,j$ with $1 \leq j \leq n,$ we obtain that
		\begin{multline}
			\label{eq:sumj}
			\rho c \tq^\alpha \sum_{i =1}^{j} \scal{\delta\left(\galpha \ast \left( \delta u-V_0\right)\right)^{c}_i}{\delta u_i}\tau + a \tq^\alpha \sum_{i =1}^{j} \scal{\left(\galpha \ast \delta u\right)^{c}_i}{\delta u_i}\tau \\ + \rho c \sum_{i =1}^{j} \nrm{\delta u_i}^2 \tau + \sum_{i=1}^{j} \mathcal{L}\scal{u_i}{\delta u_i}\tau  = \sum_{i=1}^{j} \scal{F_i}{\delta u_i}\tau.
		\end{multline}
		The proof now proceeds by handling each term of the above expression. 
		The first convolution term on the left-hand side of \eqref{eq:sumj} can be rewritten as
		\begin{multline}
			\label{eq:sumj_first_term}
			\sum_{i =1}^{j} \scal{\delta\left(\galpha \ast \left( \delta u-V_0\right)\right)^{c}_i}{\delta u_i}\tau = \sum_{i =1}^{j} \scal{\delta\left(\galpha \ast \left( \delta u-V_0\right)\right)^{c}_i}{\delta u_i-V_0}\tau \\
			+  \scal{\left(\galpha \ast \left( \delta u-V_0\right)\right)^{c}_j}{V_0}. 
		\end{multline}
		The first term on the right-hand side of \eqref{eq:sumj_first_term} is bounded from below by
		\begin{equation*}
			\label{eq:lowbound1}
			\sum_{i =1}^{j} \scal{\delta\left(\galpha \ast \left( \delta u-V_0\right)\right)^{c}_i}{\delta u_i-V_0}\tau \geq \frac{1}{2} \left(\galpha \ast \nrm{\delta u -V_0}^2\right)^{c}_j, 
		\end{equation*}
		as an application of Lemma~\ref{lem:lemconv} shows.
		The last term of \eqref{eq:sumj_first_term}  is moved to the right-hand side of \eqref{eq:sumj}, where it is bounded by the Cauchy Schwarz and  $\varepsilon$-Young inequality as follows
		\begin{equation*}
			\abs{\scal{\left(\galpha \ast \left( \delta u-V_0\right)\right)^{c}_j}{V_0}}
			\le  C_{\varepsilon_1} \nrm{V_0}^2 + \varepsilon_1 \nrm{\left(\galpha \ast \left( \delta u-V_0\right)\right)^{c}_j}^2.
		\end{equation*}
		Since
		\[
		\sum_{l=1}^j (g_\alpha)_{j-l+1}\tau = \sum_{l=1}^j g_\alpha (t_j-t_{l-1})\tau \leq \int_0^{t_j}g_\alpha(t_j-s) \ds \leq \nrm{\galpha}_{\Leb^1(0,T)}, 
		\]
		we obtain that
		\begin{equation*}
			\abs{\scal{\left(\galpha \ast \left( \delta u-V_0\right)\right)^{c}_j}{V_0}}
			\le   C_{\varepsilon_1}  + \varepsilon_1  \left(\galpha \ast \nrm{\delta u -V_0}^2\right)^{c}_j,
		\end{equation*}
		where the constant $C_{\varepsilon_1}$ only depends on $\nrm{V_0}.$ 
		By \cite[Eq. (3.2)]{Slodicka1997}, we have that
		\begin{equation*}
			\label{eq:lowbound2} 
			\sum_{i=1}^{j} \scal{\left(\galpha \ast \delta u\right)^{c}_i}{\delta u_i}\tau  \geq 0.
		\end{equation*}
		For the term involving the bilinear form $\mathcal{L}$,  we use that $\matrix{k}$ is uniformly elliptic and symmetric. Lemma~\ref{lem:exPP} shows that
		\begin{align*}
			\label{eq:lowboundLk}
			\sum_{i=1}^{j} \scal{\matrix{k} \nabla u_i}{\delta\nabla u_i}\tau &= \frac{1}{2} \scal{\matrix{k} \nabla u_j}{\nabla u_j} - \frac{1}{2} \scal{\matrix{k} \nabla U_0}{\nabla U_0}\nonumber\\ & \quad + \frac{1}{2} \sum_{i=1}^{j} \scal{\matrix{k}\left(\nabla u_i - \nabla u_{i-1}\right)}{\nabla u_i-\nabla u_{i-1}} \nonumber\\
			& \geq \frac{\widetilde{k}}{2} \nrm{\nabla u_j}^2 -C + \frac{\widetilde{k}}{2} \sum_{i\leq j} \nrm{\nabla u_i - \nabla  u_{i-1}}^2
		\end{align*}
		and (as $a\geq 0$)
		\begin{equation*} 
			\label{eq:lowboundLa}
			a \sum_{i=1}^{j}\scal{u_i}{\delta u_i}\tau = \frac{a}{2}\left(\nrm{u_j}^2 - \nrm{U_0}^2 + \sum_{i=1}^{j} \nrm{u_i- u_{i-1}}^2 \right) \geq -\frac{a}{2} \nrm{U_0}^2,
		\end{equation*}
		where $C$ depends on $\nrm{\matrix{k}}_{\Lp{\infty}}$ and $\nrm{U_0}_{\hk{1}}.$
		
		The right-hand side of \eqref{eq:sumj} is estimated by use of the $\varepsilon$-Young inequality as
		\begin{equation*}
			\label{eq:boundF}
			\abs{\sum_{i=1}^{j} \scal{F_i}{\delta u_i} \tau} \leq C_{\varepsilon_2} + \varepsilon_2 \sum_{i=1}^{j} \nrm{\delta u_i}^2 \tau,
		\end{equation*}
		where $C_{\varepsilon_2}$ only depends on $\nrm{F}_{
			\Leb^2\left([0,T],\lp{2}\right)}.$
		
		Summarising the above results, we obtain the estimate
		\begin{multline*}
			\rho c \tq^\alpha\left( \frac{1}{2} -\veps_1\right) \left( \galpha \ast \nrm{\delta u-V_0}^2\right)^{c}_j + \left(\rho c  -\varepsilon_2\right) \sum_{i=1}^{j} \nrm{\delta u_i}^2 \tau \\
			+ \frac{\widetilde{k}}{2} \nrm{\nabla u_j}^2 + \frac{\widetilde{k}}{2} \sum_{i=1}^{j} \nrm{\nabla u_i - \nabla u_{i-1}}^2 
			\leq C + C_{\varepsilon_1} + C_{\varepsilon_2}.
		\end{multline*}
		Fixing $\varepsilon_1$ and $\varepsilon_2$ sufficiently small and applying the Friedrichs' inequality yield the stated result. 
	\end{proof}
	
	\begin{lemma} \label{lem:boundsDY}
		Let the assumptions of Lemma~\ref{lem:1} be fulfilled. Then, there exists a positive constant $C$ such that
		\begin{equation*}
			\label{eq:bound0} 
			\sum_{j=1}^n \nrm{\left(\galpha \ast \delta u \right)^{c}_j}^2 \tau  \leq C.
		\end{equation*}
		\begin{proof} 
			Note that by the discrete Young inequality, we have that 
			\begin{align*}
				\sum_{j=1}^n \nrm{\left(\galpha \ast \delta u\right)^{c}_j}^2 \tau
				& = \int_\Omega \sum_{j=1}^n \abs{\left(\galpha \ast \delta u\right)^{c}_j(\X)}^2 \tau \di\X  \\ 
				& \leq  \left(\sum_{j=0}^{n-1} \left(\galpha\right)_{j+1} \tau\right)^2 \sum_{j=1}^n \nrm{\delta u_j}^2\tau \\
				&\leq C \nrm{\galpha}^2_{\Leb^1(0,T)} 
			\end{align*} 
			by means of Lemma~\ref{lem:1}.
		\end{proof} 
	\end{lemma} 
	
	\begin{lemma}
		\label{lem:bound3v2}
		Let the assumptions of Lemma~\ref{lem:1} be satisfied. Then, there exists a positive constant $C$ such that
		\begin{equation*}
			\label{eq:bound3v2}
			\sum_{j=1}^n \nrm{\delta\left(\galpha \ast (\delta u-V_0)\right)^{c}_j}_{{\hko{1}}^\ast}^2 \tau \leq C.
		\end{equation*}
	\end{lemma} 
	\begin{proof} 
		Note that
		\begin{align*} 
			\nrm{\delta\left(\galpha \ast (\delta u-V_0)\right)^{c}_j}_{{\hko{1}}^\ast} &= \sup\limits_{\nrm{\varphi}_{\hko{1}}= 1} \abs{\inpd{\delta\left(\galpha \ast (\delta u-V_0)\right)^{c}_j}{\varphi}}\nonumber\\
			&=\frac{1}{\rho c \tq^\alpha} \sup\limits_{\nrm{\varphi}_{\hko{1}}=1} \left\lvert\scal{F_j}{\varphi} - \mathcal{L}\left(u_j,\varphi\right) \right.\nonumber \\
			& \qquad \qquad \qquad - \rho c \scal{\delta u_j}{\varphi} -a\tq^\alpha \left.\scal{\left(\galpha \ast \delta u\right)_j}{\varphi}\right\rvert\nonumber \\
			& \leq C \left(\nrm{F_j} + \nrm{u_j}_{\hk{1}} + \nrm{\delta u_j} + \nrm{\left(\galpha \ast \delta u\right)^{c}_j} \right).\label{eq:b3v2}
		\end{align*}
		Applying the rule $(a+b)^2 \leq 2a^2 + 2b^2$ yields
		\begin{multline*} 
			\sum_{j = 1}^n \nrm{\delta\left(\galpha \ast (\delta u-V_0)\right)^{c}_j}^2 \tau \\ \leq C \left(\sum_{j=1}^n \nrm{F_j}^2\tau + \sum_{j=1}^n \nrm{u_j}^2_{\hk{1}} \tau  + \sum_{j= 1}^n \nrm{\delta u_j}^2 \tau + \sum_{j=1}^n \nrm{\left(\galpha \ast \delta u\right)^{c}_j}^2\tau \right),
		\end{multline*}
		which is uniformly bounded by Lemma~\ref{lem:1} and \ref{lem:boundsDY}.
	\end{proof}

	\subsection{Existence} \label{subsec_rothe_conv} 
	
	The existence of a weak solution in shown in this subsection. To achieve this goal, we introduce the Rothe functions $[0,T]\to \lp{2}$, which are build from the solutions $u_i, i = 0,\dots,n$ at each time slice as follows,
	\begin{align*}
		v_n \colon t &\mapsto \begin{cases}
			U_0, & t =0 \\ u_{i-1} + (t-t_{i-1}) \delta u_i, & t \in (t_{i-1},t_i], \quad  1 \leq i \leq n,
		\end{cases}\\
		\overline{v}_n \colon t & \mapsto \begin{cases}
			U_0, & t = 0 \\ u_i, & t \in (t_{i-1},t_i], \quad 1 \leq i \leq n,
		\end{cases} \\
		\overline{w}_n \colon t& \mapsto \begin{cases}
			V_0, & t = 0 \\ \delta u_i, & t\in(t_{i-1},t_i], \quad 1 \leq i \leq n. 
		\end{cases}
		\end{align*}
		
		Note that $\pdt v_n = \overline{w}_n.$ Similar notations are used for $\overline{F}_n$ and $ \overline{\ggamma}.$ 
		Additionally, we define for two functions $\kappa$ and $z$ \cite{VanBockstal2020}:
		\begin{align*}
			C_{\kappa ,z}^n \colon t &\mapsto \begin{cases}
				0, & t =  0  \\
				\left(\kappa\ast z \right)^{c}_{i-1} + (t-t_{i-1})\delta \left(\kappa\ast z \right)^{c}_i, & t \in (t_{i-1}, t_i],
			\end{cases}\\
			\overline{C}_{\kappa,z}^n \colon t & \mapsto \begin{cases}
				0, & t = 0 \\ 
				\left(\kappa\ast z \right)^{c}_{i}, & t \in (t_{i-1},t_i].
			\end{cases} 
		\end{align*}
		Note that for $t \in (t_{i-1},t_i]$,  we have that
		\[
		\pdt \left(C^n_{\overline{\galpha}_n, \overline{w}_n-V_0} \right)(t)  = \delta \left(\overline{\galpha}_n \ast (\overline{w}_n-V_0)\right)^{c}_i  = \delta\left(\galpha \ast (\delta u-V_0)\right)^{c}_i
		\]
		as $\overline{w}_n(t_k) = \delta u_k.$ 
		Moreover, we have that
		\begin{align*}
			\left(\overline{\galpha}_{n} \ast \overline{w}_{n}\right)(\lceil t\rceil_\tau) = \int_0^{t_i} \overline{\galpha}_{n}(t_i-s) \overline{w}_{n}(s)\ds  & = \sum_{j = 1}^i \int_{t_{j-1}}^{t_j} \overline{\galpha}_{n}(t_i - s)\delta u_j \ds \\ & = \sum_{j=1}^{i } \galpha(t_{i-j+1}) \delta u_j \tau
		\end{align*}
		as for $s\in (t_{j-1},t_j]$, we have $t_i - s \in [t_{i-j}, t_{i-j+1})$ and so $\overline{\galpha}_{n}(t_i-s) = \galpha(t_{i-j+1}).$
		
		Now, we extend \eqref{eq:VFI} to the whole time frame. In terms of the Rothe functions, the discrete variational formulation now reads as
		\begin{multline}
			\label{eq:extrothe_v2}
			\rho c \tq^\alpha \scal{\pdt \left(C^n_{\overline{\galpha}_n, \overline{w}_n-V_0}\right)(t) }{\varphi} + a\tq^\alpha \scal{(\overline{\galpha}_n \ast \overline{w}_n)(\lceil t\rceil_{\tau}) }{\varphi} \\ + \rho c \scal{\pdt v_n(t)}{\varphi} + \mathcal{L}\scal{\overline{v}_n(t)}{\varphi} = \scal{\overline{F}_n(t)}{\varphi}, \quad \forall \varphi \in \hko{1}.
		\end{multline}

		In the remainder of this subsection, we show that a subsequence of the Rothe functions converges to a weak solution. The following subsection is then devoted to proving the uniqueness of a solution.
		
		\begin{theorem}[Existence (SPL)] \label{thm:existence spl}
			Suppose that $F \in \Leb^2\left([0,T],\lp{2}\right)$, $U_0 \in \hko{1}$ and $V_0 \in \lp{2}.$ Then, there exists a weak solution $u$ to \eqref{eq:VF} with $u \in \lpkIX{\infty}{\hko{1}}\cap \cIX{\lp{2}}$ for which $\pdt u \in \lpIlp$ and  $\pdt\left(\galpha \ast \left(\pdt u - V_0 \right)\right) \in \lpkIX{2}{{\hko{1}}^\ast}.$ 
		\end{theorem} 
		
		\begin{proof} 
			The Rellich-Kondrachov theorem \cite[Theorem~6.3]{Necas2011} gives us the compact embedding of $\hko{1}$ in $\lp{2}.$ The estimates from Lemma~\ref{lem:1} show in particular that
			\begin{equation*}
				\nrm{\overline{v}_n(t)}_{\hko{1}} \leq C, \quad \forall t \in [0,T],
			\end{equation*}
			and 
			\begin{equation*}
				\int_0^T \nrm{\pdt v_n(t)}_{\lp{2}}^2 \dt  = \sum_{i=1}^n \nrm{\delta u_i}_{\lp{2}}^2 \tau \leq C.
			\end{equation*}
			Hence, by the Aubin-Lions lemma \cite[Lemma~1.3.13]{Kavcur1985}, we find a function $v \in \cIX{\lp{2}} \cap \lpkIX{\infty}{\hko{1}}$ with $\pdt v \in \lpIlp$ and a corresponding subsequence $\{v_{n_k}\}_{k \in \NN}$ of $\{v_n\}_{n\in\NN}$ having the following properties, as $k \to \infty,$
			\begin{equation*}
				\begin{array}{rl}
					v_{n_k} \to v & \text{ in } \cIX{\lp{2}}\\
					v_{n_k}(t) \wto v(t) & \text{ in } \hko{1}, \text{ for all } t \in [0,T]\\
					\overline{v}_{n_k}(t) \wto v(t) & \text{ in } \hko{1}, \text{ for all } t \in [0,T]\\
					\pdt v_{n_k} =  \overline{w}_{n_k} \wto \pdt v & \text{ in } \lpIlp.
				\end{array}
			\end{equation*}
			Note that 
			\begin{align*}
				\int_0^T \nrm{\nabla v_{n_k}(t) - \nabla \overline{v}_{n_k}(t)}^2\dt & = \sum_{i=1}^{n_k} \int_{t_{i-1}}^{t_i} \nrm{(t-t_i)\delta \nabla u_i}^2 \dt \\
				&= \frac{\tau_{n_k}}{3} \sum_{i=1}^{n_k} \nrm{\nabla\delta u_i}^2 \tau_{n_k}^2 \\ 
				& = \frac{\tau_{n_k}}{3} \sum_{i=1}^{n_k} \nrm{\nabla u_i - \nabla u_{i-1}}^2  \leq C \tau_{n_k},
			\end{align*}
			where we have set $\tau_{n_k} = T/n_k.$ As a result, $\{v_{n_k}\}_{k\in \NN}$ and $\{\overline{v}_{n_k}\}_{k\in \NN}$ have the same limit $v$ in $\lpkIX{2}{\hko{1}}.$ 
			
			Next, we integrate \eqref{eq:extrothe_v2} over the interval $(0,\eta) \subset (0,T)$ to arrive at
			\begin{multline}
				\label{eq:ex_rothe_int}
				\rho c \tq^\alpha\scal{ C^{n_k}_{\overline{\galpha}_{n_k},\overline{w}_{n_k}-V_0}(\eta)}{\varphi}  + a\tq^\alpha\int_0^\eta \scal{(\overline{\galpha}_{n_k} \ast \overline{w}_{n_k})(\lceil t\rceil_{\tau}) }{\varphi} \dt    \\ + \rho c \int_0^\eta \scal{\pdt v_{n_k}(t)}{\varphi}\dt + \int_0^\eta \mathcal{L}\scal{\overline{v}_{n_k}(t)}{\varphi}\dt  = \int_0^\eta \scal{\overline{F}_{n_k}(t)}{\varphi}\dt. 
			\end{multline}
			Now, we will first discuss some convergence results. Using Lemma~\ref{lem:bound3v2}, we can estimate
			\begin{align*}
				&\abs{ \int_0^T \scal{C^{n_k}_{\overline{\galpha}_{n_k}, \overline{w}_{n_k}-V_0}(t) - \overline{C}^{n_k}_{\overline{\galpha}_{n_k},\overline{w}_{n_k}-V_0}(t)}{\varphi}\dt} \\&= \sum_{i = 1}^{n_k} \int_{t_{i-1}}^{t_i} \abs{\scal{(t-t_i)\delta\left(\galpha \ast (\delta u-V_0)\right)^{c}_i}{\varphi}} \dt \\ 
				&\leq \sum_{i=1}^{n_k} \tau_{n_k}^2\abs{\scal{ \delta\left(\galpha \ast (\delta u-V_0)\right)^{c}_i}{\varphi}}\dt \\
				& \leq  C\tau_{n_k}\sqrt{\sum_{i=1}^{n_k} \nrm{\delta\left(\galpha \ast (\delta u-V_0)\right)^{c}_i}_{{\hk{1}}^{*}}^2 \tau_{n_k}} \\
				& \leq C \tau_{n_k} \to 0 \quad \text{ as } k \to \infty,
			\end{align*}
			We note that
			$$
			\overline{C}^{n_k}_{\overline{\galpha}_{n_k}, \overline{w}_{n_k}-V_0}(t) = \sum_{j=1}^i \galpha(t_{i+1-j}) (\delta u_j-V_0) \tau_{n_k} = \left(\overline{\galpha}_{n_k} \ast (\overline{w}_{n_k}-V_0)\right) \left(\lceil t\rceil_\tau\right).
			$$
			The same arguments as in \cite[p. 22--25]{VanBockstal2021} yield the limit transition
			\begin{equation*}
				\abs{  \int_0^\xi \scal{\overline{C}^{n_k}_{\overline{\galpha}_{n_k}, \overline{w}_{n_k}-V_0}(\eta) - \left(\galpha \ast (\overline{w}_{n_k}-V_0)\right)(\eta)}{\varphi}   \di\eta } \to 0, \quad \text{ as } k \to \infty,
			\end{equation*}
			for any $\xi \in(0,T].$ It is here used that  $\overline{\galpha}_{n_k} \to \galpha$ pointwise in $(0,T)$ and that the estimate $\left( \galpha \ast \nrm{\delta u-V_0}^2\right)^{c}_j + \sum_{i=1}^j \nrm{\delta u_i}^2\tau_{n_k} \leq C$ is available (see Lemma~\ref{lem:1}).
			The operator
			\begin{equation*} 
				w \mapsto \int_0^\eta \scal{\left(\ggamma \ast w\right)(t)}{\varphi}\dt,
			\end{equation*} 
			with $\gamma \in (0,1)$ and $\eta \in(0,T)$ fixed, is a bounded linear functional on $\lpIlp$ since by Young's inequality for convolutions it holds that
			\begin{equation*}
				\abs{\int_0^\eta \scal{\left(g_{\gamma} \ast w\right)(t)}{\varphi}\dt } \leq \sqrt{T} \nrm{g_{\gamma}}_{\Leb^1\Iopen} \nrm{\varphi} \nrm{w}_{\lpIlp}. 
			\end{equation*}  
			Therefore, using $\overline{w}_{n_k} \wto \pdt v$ in $\lpIlp$ it finally holds that 
			\begin{equation*}
				\abs{  \int_0^\xi \scal{{C}^{n_k}_{\overline{\galpha}_{n_k}, \overline{w}_{n_k}-V_0}(\eta) - \left(\galpha \ast (\pdt v-V_0)\right)(\eta)}{\varphi}   \di\eta } \to 0, \quad \text{ as } k \to \infty,
			\end{equation*}
			for any $\xi \in(0,T).$
			Next, we note that
			\[
			\overline{C}^{n_k}_{\overline{\galpha}_{n_k}, \overline{w}_{n_k}}(t) = \left(\overline{\galpha}_{n_k} \ast \overline{w}_{n_k}\right)(\lceil t\rceil_{\tau_{n_k}})
			\]
			Hence,  we immediately see that 
			\[ 
			\abs{\int_0^\xi \scal{\overline{C}^{n_k}_{\overline{\galpha}_{n_k}, \overline{w}_{n_k}}(t) - \left( \galpha \ast \pdt v \right)(t)}{\varphi} \dt} \to 0 \quad \text{ as } k\to \infty. 
			\] 
			Furthermore, we have that
			\begin{align*}
				\abs{\int_0^\eta \mathcal{L}\scal{\overline{v}_{n_k}(t)}{\varphi} \dt - \int_0^\eta \mathcal{L}\scal{v(t)}{\varphi}\dt } &\to 0 \quad \text{ as } k \to \infty, \\
				\abs{\int_0^\eta \scal{\overline{F}_{n_k}(t)}{\varphi}\dt  - \int_0^\eta \scal{F(t)}{\varphi}\dt}& \to 0 \quad \text{ as } k \to \infty.
			\end{align*}
			
			Before making the limit transition, we need to integrate \eqref{eq:ex_rothe_int} in time over $\eta \in (0,\xi) \subset (0,T)$. We obtain that
			\begin{multline}
				\label{eq:ex_rothe_int2}
				\rho c \tq^\alpha \int_0^\xi \scal{ C^{n_k}_{\overline{\galpha}_{n_k},\overline{w}_{n_k}-V_0}(\eta)}{\varphi} \di\eta + a\tq^\alpha \int_0^\xi \int_0^\eta \scal{(\overline{\galpha}_{n_k} \ast \overline{w}_{n_k})(\lceil t\rceil_{\tau}) }{\varphi} \dt\di \eta\\ + \rho c \int_0^\xi  \int_0^\eta \scal{\pdt v_{n_k}(t)}{\varphi}\dt \di \eta  
				+ \int_0^\xi \int_0^\eta \mathcal{L}\scal{\overline{v}_{n_k}(t)}{\varphi}\dt \di\eta\\
				= \int_0^\xi \int_0^\eta \scal{\overline{F}_{n_k}(t)}{\varphi}\dt \di\eta.
			\end{multline}
			Using the previously obtained results, we can make the limit transition in \eqref{eq:ex_rothe_int2} to receive that
			\begin{multline}
				\label{eq:ex_rothe_int_lim}
				\rho c \tq^\alpha \int_0^\xi \scal{\left(\galpha \ast (\pdt v-V_0)\right)(\eta)}{\varphi} \di \eta  + a\tq^\alpha \int_0^\xi \int_0^\eta \scal{\left(\galpha \ast \pdt v\right)(t)}{\varphi} \dt \di \eta \\+ \rho c  \int_0^\xi  \int_0^\eta \scal{\pdt v(t)}{\varphi}\dt \di \eta    
				+ \int_0^\xi \int_0^\eta \mathcal{L}\scal{v(t)}{\varphi}\dt \di\eta\\ = \int_0^\xi \int_0^\eta \scal {F(t)}{\varphi}\dt \di\eta.
			\end{multline}
			Differentiation of \eqref{eq:ex_rothe_int_lim} with respect to $\xi$ yields
			\begin{multline}
				\label{eq:ex_rothe_int_lim_diff}
				\rho c \tq^\alpha \scal{\left(\galpha \ast (\pdt v-V_0)\right)(\xi)}{\varphi}
				+a \tq^\alpha \int_0^\xi  \scal{\left(\galpha \ast \pdt v)\right)(t)}{\varphi} \dt\\  + \rho c  \int_0^\xi \scal{\pdt v(t)}{\varphi}\dt + \int_0^\xi \mathcal{L}\scal{v(t)}{\varphi}\dt  = \int_0^\xi \scal{F(t)}{\varphi}\dt.
			\end{multline}
			From \eqref{eq:ex_rothe_int_lim_diff}, it follows that $\lim_{\xi \searrow 0} \scal{\left(\galpha \ast (\pdt v - V_0)\right)(\xi)}{\varphi} = 0. $
			Hence, differentiating \eqref{eq:ex_rothe_int_lim_diff} again w.r.t. $\xi$ we obtain that 
			\begin{multline*}
				\rho c \tq^\alpha \inpd{ \pdt \left(\galpha \ast \left(\pdt v- V_0\right)\right)(\xi)}{\varphi} + a\tq^\alpha \scal{\left(\galpha \ast \pdt v\right)(\xi)}{\varphi} \\ + \rho c \scal{\pdt v(\xi)}{\varphi} + \mathcal{L}\scal{v(\xi)}{\varphi} = \scal{F(\xi)}{\varphi}.
			\end{multline*}
			Therefore, $v$ satisfies equation \eqref{eq:VF}.
		\end{proof} 
		\subsection{Uniqueness} \label{sec_unique}
		
	In this subsection, we show the uniqueness of the solution to problem \eqref{eq:problem} under the additional assumption that $\pdt u \in \lpkIX{\infty}{\lp{2}}$.
This assumption is reasonable considering the regularity of the solution to the SPL problem in one dimension, see Subsection~\ref{subsec_sepvar}. 

\begin{theorem}[Uniqueness SPL]\label{thm:uniqueness}%
			The weak solution to problem \eqref{eq:VF} satisfying $u \in \lpkIX{\infty}{\hko{1}}\cap \cIX{\lp{2}}$, $\pdt u \in \lpkIX{\infty}{\lp{2}}$ and  $\pdt\left(\galpha \ast \left(\pdt u - V_0 \right)\right) \in \lpkIX{2}{{\hko{1}}^\ast}$  is unique.
\end{theorem}
		\begin{proof}
			Let $u_1$ and $u_2$ be solutions to the problem \eqref{eq:VF} and consider their difference $u = u_1 - u_2.$ Then $u$ is a solution to the homogeneous problem (i.e. $F = 0$ and $U_0 = 0 = V_0$) with vanishing Dirichlet boundary conditions. By Lemma~\ref{lem:pdtconv}, the function $u$ satisfies
			\begin{multline} 
				\label{eq:exweaku}
				\rho c\tq^\alpha \inpd{\pdt \left(\galpha \ast \pdt u\right)(t)}{\varphi} + a \tq^\alpha \scal{\pdt \left(\galpha \ast u \right)(t)}{\varphi} \\ + \rho c \scal{\pdt u(t)}{\varphi} + \mathcal{L}\left(u(t),\varphi\right) = 0,
			\end{multline}
			for all $\varphi \in \hko{1}.$ Note that we can not choose $\varphi = \pdt u(t)$ as $\pdt u(t) \not\in \hko{1}.$ We use a method due to Ladyzhenskaya (see \cite{Ladyzenskaya1958} and \cite[p.202]{Valli2020}). Fix $s\in(0,T)$ and let $\varphi = v(t)$ where $v(t)$ is given by
			\[ 
			v(t) = \begin{cases}
				\int_t^s u(\tau)\dtau & \text{ if } 0 \leq t \leq s \\
				0 &\text{ otherwise}.
			\end{cases} 
			\] 
			The crucial property of $v$ is that $v^\prime(t) = -u(t)$ if $t\leq s.$ We substitute $\varphi = v(t)$ in \eqref{eq:exweaku} and integrate $t$ over $(0,s)$ to obtain that
			\begin{multline*}
				\rho c \tq^\alpha \int_0^s \inpd{\pdt\left(g_{\alpha}\ast \pdt u\right)(t)}{\int_t^s u(\tau)\dtau}\dt  \\ + a\tq^\alpha \int_0^s \scal{\pdt\left(g_{\alpha} \ast u \right)(t)}{\int_t^s u(\tau)\dtau}\dt+ \rho c \int_0^s\scal{\pdt u(t)}{\int_t^s u(\tau)\dtau}\dt\\ + \int_0^s \mathcal{L}\left(u(t), \int_t^s u(\tau)\dtau\right) = 0.
			\end{multline*}
The additional assumption on $\pdt u$ implies that  
\begin{equation*}
\lim\limits_{t\underset{>}{\rightarrow}0} \nrm{\left(\galpha \ast \pdt u\right)(t)} \le \lim\limits_{t\underset{>}{\rightarrow}0} \left(\galpha \ast \nrm{\pdt u}\right)(t)  \le C\lim\limits_{t\underset{>}{\rightarrow}0} \int_0^t (t-s)^{-\alpha}  = 0.
\end{equation*}
			Hence, partial integration in the first fractional term yields
			\begin{align*}
				& \int_0^s \inpd{\pdt \left(\galpha \ast \pdt u\right)(t)}{\int_t^s u(\tau)\dtau}\dt \\
				& \quad = \left.\scal{\left(\galpha \ast \pdt u\right)(t)}{\int_t^s u(\tau)\dtau}\right\rvert_{t=0}^{t=s} + \int_0^s \scal{\left(\galpha \ast \pdt u\right)(t)}{u(t)}\dt \\ 
				& \quad = \int_0^s \scal{\pdt\left(\galpha \ast u\right)(t)}{u(t)}\dt \\
				& \quad \geq \frac{T^{-\alpha}}{2\Gamma(1-\alpha)}\int_0^s \nrm{u(t)}^2\dt, 
			\end{align*}
			where the last inequality follows from Lemma~\ref{lemma:g}(v). 
			Analogously, for the second fractional term, partial integration shows that
			\begin{equation*} 
				\int_0^s \scal{\pdt\left(\galpha \ast u\right)(t)}{\int_t^s u(\tau)\dtau}\dt = \int_0^s \scal{\left(\galpha \ast u\right)(t)}{u(t)}\dt \geq 0,
			\end{equation*} 
			which is positive by Lemma~\ref{lemma:g}(iv). 
			Since
			\begin{align*}
				\frac{\mathrm{d}}{\dt} \scal{u(t)}{\int_t^s u(\tau)\dtau} &= \scal{\pdt u(t)}{\int_t^s u(\tau)\dtau} + \scal{u(t)}{\pdt \int_t^s u(\tau)\dtau }\\ &= \scal{\pdt u(t)}{\int_t^s u(\tau)\dtau} - \nrm{u(t)}^2,
			\end{align*}
			we find that
			\begin{align*}
				\int_0^s  \scal{\pdt u(t)}{\int_t^s u(\tau)\dtau}\dt &= \int_0^s \frac{\mathrm{d}}{\dt} \scal{u(t)}{\int_t^s u(\tau)\dtau}\dt + \int_0^s \nrm{u(t)}^2\dt \\ &=
				- \scal{u(0)}{\int_0^s u(\tau)\dtau} + \int_0^s \nrm{u(t)}^2\dt,
			\end{align*}
			and therefore
			\begin{equation*}
				\rho c \int_0^s \scal{\pdt u(t)}{\int_t^s u(\tau)\dtau}\dt  = \rho c \int_0^s \nrm{u(t)}^2\dt. 
			\end{equation*}
			Next, we have that
			\begin{align*}
				a \int_0^s \scal{u(t)}{\int_t^s u(\tau)\dtau}\dt &= -a \int_0^s \scal{-u(t)}{\int_t^s u(\tau)\dtau}\dt \\ & = -\frac{a}{2}\int_0^s \pdt \nrm{\int_t^s u(\tau)\dtau}^2 \dt  \\ & = \frac{a}{2} \nrm{\int_0^s u(t)\dt}^2 \geq 0.
			\end{align*}
			Similarly, by the symmetry of $\matrix{k}$, we obtain that
			\begin{align*}
				\int_0^s \scal{\matrix{k} \nabla u(t)}{\int_t^s \nabla u(\tau)\dtau} &= -\int_0^s \scal{-\matrix{k} \nabla u(t)}{\int_t^s \nabla u(\tau)\dtau}\\ &= -\frac{1}{2} \int_0^s \pdt \scal{\matrix{k} \int_t^s \nabla u(\tau)\dtau}{\int_t^s \nabla u(\tau)\dtau} \\ & \geq \frac{\widetilde{k}}{2} \nrm{\int_0^s\nabla u(t)}^2\dt.
			\end{align*}
			Summarising, we obtain the inequality
			\[
			\rho c \left(1 + \tq^\alpha \frac{T^{-\alpha}}{2\Gamma(1-\alpha)}\right) \int_0^s \nrm{u(t)}^2\dt
			+ \frac{\widetilde{k}}{2} \nrm{\int_0^s \nabla u(t)\dt }^2 \leq 0.
			\]
			Therefore, $u= 0$ a.e. in $Q_T$, which shows that $u_1 = u_2$ a.e. in $Q_T$. 
		\end{proof}
\begin{remark}
From theoretical viewpoint, Theorem~\ref{thm:uniqueness} stays valid when the condition on $\nrm{\pdt u(t)}$ is replaced by 
	\[
 \nrm{\pdt u(t)} \le C \left(1+t^{-\beta}\right),\quad 0 < \beta < \min\left\{\half,1-\alpha\right\}.
\]
\end{remark}
		
		\section{Conclusion} \label{sec_conclusion} 
In this article, we have investigated the existence and uniqueness of a weak solution to the fractional single-phase-lag heat equation. Rothe's method has been employed to show the existence of a solution satisfying $u \in \lpkIX{\infty}{\hko{1}}\cap \cIX{\lp{2}}$ and $\pdt u \in \lpkIX{2}{\lp{2}}$. Under the additional assumption $\pdt u \in \lpkIX{\infty}{\lp{2}}$, the uniqueness of a solution has been shown by contradiction. Moreover, we have derived an explicit solution to the one-dimensional problem. Bounds on the solution and its time derivative have been obtained by generalising the uniform boundedness property of the multinomial Mittag-Leffler function. Future work will concern the existence of a unique solution to the DPL model corresponding with \eqref{eq:tzoumodel_simp}. 
		
		\bibliography{refs}

\begin{thebibliography}{10}

\bibitem{AtanackovicPilipovic2014book}
T.~M. Atanackovi\'{c}, S.~Pilipovi\'{c}, B.~Stankovi\'{c}, and D.~Zorica.
\newblock {\em Fractional calculus with applications in mechanics. Wave
  propagation, impact and variational principles}.
\newblock Mechanical Engineering and Solid Mechanics Series. ISTE, London; John
  Wiley \& Sons, Inc., Hoboken, NJ, 2014.

\bibitem{Bazhlekova2021}
E.~Bazhlekova.
\newblock Completely monotone multinomial {M}ittag-{L}effler type functions and
  diffusion equations with multiple time-derivatives.
\newblock {\em Fract. Calc. Appl. Anal.}, 24(1):88--111, 2021.

\bibitem{Caputo1967}
M.~Caputo.
\newblock {Linear Models of Dissipation whose Q is almost Frequency
  Independent—II}.
\newblock {\em Geophysical Journal International}, 13(5):529--539, 11 1967.

\bibitem{Chen2008}
J.~Chen, F.~Liu, and V.~Anh.
\newblock Analytical solution for the time-fractional telegraph equation by the
  method of separating variables.
\newblock {\em J. Math. Anal. Appl.}, 338(2):1364--1377, 2008.

\bibitem{Luchko1996}
S.~B. Hadid and Y.~F. Luchko.
\newblock An operational method for solving fractional differential equations
  of an arbitrary real order.
\newblock {\em PanAmer. Math. J.}, 6(1):57--73, 1996.

\bibitem{Hosseini2014}
V.~R. Hosseini, W.~Chen, and Z.~Avazzadeh.
\newblock Numerical solution of fractional telegraph equation by using radial
  basis functions.
\newblock {\em Eng. Anal. Bound. Elem.}, 38:31--39, 2014.

\bibitem{Kavcur1985}
J.~Ka\v{c}ur.
\newblock {\em Method of {R}othe in evolution equations}, volume~80 of {\em
  Teubner-Texte zur Mathematik [Teubner Texts in Mathematics]}.
\newblock BSB B. G. Teubner Verlagsgesellschaft, Leipzig, 1985.
\newblock With German, French and Russian summaries.

\bibitem{Kian2017}
Y.~{Kian} and M.~{Yamamoto}.
\newblock {On existence and uniqueness of solutions for semilinear fractional
  wave equations.}
\newblock {\em {Fractional Calculus and Applied Analysis}}, 20(1):117--138,
  2017.

\bibitem{Kubica2017}
A.~{Kubica}, P.~{Rybka}, and K.~{Ryszewska}.
\newblock {Weak solutions of fractional differential equations in non
  cylindrical domains}.
\newblock {\em {Nonlinear Anal., Real World Appl.}}, 36:154--182, 2017.

\bibitem{Kubica2018}
A.~Kubica and M.~Yamamoto.
\newblock Initial-boundary value problems for fractional diffusion equations
  with time-dependent coefficients.
\newblock {\em Fract. Calc. Appl. Anal.}, 21(2):276--311, 2018.

\bibitem{KUMAR201749}
D.~Kumar and K.~N. Rai.
\newblock Numerical simulation of time fractional dual-phase-lag model of heat
  transfer within skin tissue during thermal therapy.
\newblock {\em Journal of Thermal Biology}, 67:49--58, 2017.

\bibitem{Kumar2014}
S.~Kumar.
\newblock A new analytical modelling for fractional telegraph equation via
  {Laplace} transform.
\newblock {\em Appl. Math. Modelling}, 38(13):3154--3163, 2014.

\bibitem{Ladyzenskaya1958}
O.~A. Lady\v{z}enskaya.
\newblock On integral estimates, convergence, approximate methods, and solution
  in functionals for elliptic operators.
\newblock {\em Vestnik Leningrad. Univ.}, 13(7):60--69, 1958.

\bibitem{LiLi2020}
C.~Li and C.~Li.
\newblock The fractional {G}reen's function by {B}abenko's approach.
\newblock {\em Tbilisi Math. J.}, 13(3):19--42, 2020.

\bibitem{LiLiuYamamoto2015}
Z.~Li, Y.~Liu, and M.~Yamamoto.
\newblock Initial-boundary value problems for multi-term time-fractional
  diffusion equations with positive constant coefficients.
\newblock {\em Appl. Math. Comput.}, 257:381--397, 2015.

\bibitem{Liu2013}
F.~Liu, M.~M. Meerschaert, R.~J. McGough, P.~Zhuang, and Q.~Liu.
\newblock Numerical methods for solving the multi-term time-fractional
  wave-diffusion equation.
\newblock {\em Fractional Calculus and Applied Analysis}, 16(1):9--25, 2013.

\bibitem{Liu2020}
Z.~Liu, R.~Quintanilla, and Y.~Wang.
\newblock On the regularity and stability of the dual-phase-lag equation.
\newblock {\em Appl. Math. Lett.}, 100:8, 2020.
\newblock Id/No 106038.

\bibitem{Luchko2010}
Y.~Luchko.
\newblock Some uniqueness and existence results for the initial-boundary-value
  problems for the generalized time-fractional diffusion equation.
\newblock {\em {Computers \& Mathematics with Applications}}, 59(5):1766--1772,
  2010.
\newblock Fractional Differentiation and Its Applications.

\bibitem{Luchko2011}
Y.~Luchko.
\newblock Initial-boundary-value problems for the generalized multi-term
  time-fractional diffusion equation.
\newblock {\em Journal of Mathematical Analysis and Applications},
  374(2):538--548, 2011.

\bibitem{Luchko2012}
Y.~{Luchko}.
\newblock {Initial-boundary-value problems for the one-dimensional
  time-fractional diffusion equation.}
\newblock {\em {Fractional Calculus and Applied Analysis}}, 15(1):141--160,
  2012.

\bibitem{SlodickaMaes}
F.~Maes and M.~Slodi\v{c}ka.
\newblock Some inverse source problems of determining a space dependent source
  in fractional-dual-phase-lag type equations.
\newblock {\em Mathematics}, 8:1291, 2020.

\bibitem{Maes2020}
F.~Maes and M.~Slodi\v{c}ka.
\newblock Some inverse source problems of determining a space dependent source
  in fractional-dual-phase-lag type equations.
\newblock {\em Mathematics}, 8(8), 2020.

\bibitem{Necas2011}
J.~Ne\v{c}as.
\newblock {\em Direct methods in the theory of elliptic equations}.
\newblock Springer Monographs in Mathematics. Springer, Heidelberg, 2011.

\bibitem{Otarola2019}
E.~{Ot\'arola} and A.~J. {Salgado}.
\newblock {Regularity of solutions to space-time fractional wave equations: A
  PDE approach.}
\newblock {\em {Fract. Calc. Appl. Anal.}}, 21(5):1262--1293, 2019.

\bibitem{Peszynska1996}
M.~Peszy{\'n}ska.
\newblock Finite element approximation of diffusion equations with convolution
  terms.
\newblock {\em Math. Comput.}, 65(215):1019--1037, 1996.

\bibitem{Podlubny1998}
I.~Podlubn\'y.
\newblock {\em Fractional Differential Equations: An Introduction to Fractional
  Derivatives, Fractional Differential Equations, to Methods of Their Solution
  and Some of Their Applications}.
\newblock Mathematics in Science and Engineering. Elsevier Science, 1998.

\bibitem{Quintanilla2006}
R.~Quintanilla and R.~Racke.
\newblock A note on stability in dual-phase-lag heat conduction.
\newblock {\em International Journal of Heat and Mass Transfer},
  49(7):1209--1213, 2006.

\bibitem{Quintanilla2007}
R.~Quintanilla and R.~Racke.
\newblock Qualitative aspects in dual-phase-lag heat conduction.
\newblock {\em Proc. R. Soc. Lond., Ser. A, Math. Phys. Eng. Sci.},
  463(2079):659--674, 2007.

\bibitem{Sakamoto2011}
K.~Sakamoto and M.~Yamamoto.
\newblock Initial value/boundary value problems for fractional diffusion-wave
  equations and applications to some inverse problems.
\newblock {\em Journal of Mathematical Analysis and Applications},
  382(1):426--447, 2011.

\bibitem{Samko1993}
S.~G. {Samko}, A.~A. {Kilbas}, and O.~I. {Marichev}.
\newblock {\em {Fractional integrals and derivatives: theory and applications.
  Transl. from the Russian.}}
\newblock New York, NY: Gordon and Breach, 1993.

\bibitem{SINGH20112316}
J.~Singh, P.~K. Gupta, and K.~N. Rai.
\newblock Solution of fractional bioheat equations by finite difference method
  and {HPM}.
\newblock {\em Mathematical and Computer Modelling}, 54(9):2316 -- 2325, 2011.

\bibitem{Slodicka1997}
M.~Slodi\v{c}ka.
\newblock Numerical solution of a parabolic equation with a weakly singular
  positive-type memory term.
\newblock {\em Electron. J. Differential Equations}, pages No. 09, 12, 1997.

\bibitem{Slodicka2016}
M.~Slodi\v{c}ka and K.~\v{S}i\v{s}kov\'{a}.
\newblock An inverse source problem in a semilinear time-fractional diffusion
  equation.
\newblock {\em Computers \& Mathematics with Applications}, 72(6):1655--1669,
  2016.

\bibitem{Stojanovic2011}
M.~Stojanovi{\'c}.
\newblock Numerical method for solving diffusion-wave phenomena.
\newblock {\em J. Comput. Appl. Math.}, 235(10):3121--3137, 2011.

\bibitem{Sun2018b}
H.~Sun, Y.~Zhang, D.~Baleanu, W.~Chen, and Y.~Chen.
\newblock A new collection of real world applications of fractional calculus in
  science and engineering.
\newblock {\em Communications in Nonlinear Science and Numerical Simulation},
  64:213--231, 2018.

\bibitem{Tzou1995}
D.~Y. Tzou.
\newblock {A Unified Field Approach for Heat Conduction From Macro- to
  Micro-Scales}.
\newblock {\em Journal of Heat Transfer}, 117(1):8--16, 02 1995.

\bibitem{Valli2020}
A.~Valli.
\newblock {\em A compact course on linear {PDE}s}, volume 126 of {\em Unitext}.
\newblock Springer, Cham, [2020] \copyright 2020.
\newblock La Matematica per il 3+2.

\bibitem{VanBockstal2020}
K.~Van~Bockstal.
\newblock Existence and uniqueness of a weak solution to a non-autonomous
  time-fractional diffusion equation (of distributed order).
\newblock {\em Appl. Math. Lett.}, 109:7, 2020.
\newblock Id/No 106540.

\bibitem{VanBockstal2020b}
K.~Van~Bockstal.
\newblock Existence of a unique weak solution to a nonlinear non-autonomous
  time-fractional wave equation (of distributed-order).
\newblock {\em Mathematics}, 8(8), 2020.

\bibitem{VanBockstal2021}
K.~Van~Bockstal.
\newblock Existence of a unique weak solution to a non-autonomous
  time-fractional diffusion equation with space-dependent variable order.
\newblock {\em Advances in Difference Equations}, 314:43, 2021.

\bibitem{VanBockstal2022}
K.~Van~Bockstal, A.~S. Hendy, and M.~A. Zaky.
\newblock Space-dependent variable-order time-fractional wave equation:
  existence and uniqueness of its weak solution.
\newblock {\em Quaestiones Mathematicae}, 0(0):1--21, 2022.

\bibitem{VanBockstal2022d}
K.~{Van Bockstal}, M.~A. Zaky, and A.~S. Hendy.
\newblock On the existence and uniqueness of solutions to a nonlinear variable
  order time-fractional reaction–diffusion equation with delay.
\newblock {\em Communications in Nonlinear Science and Numerical Simulation},
  115:106755, 2022.

\bibitem{Wang2019d}
H.~Wang and X.~Zheng.
\newblock Analysis and numerical solution of a nonlinear variable-order
  fractional differential equation.
\newblock {\em Advances in Computational Mathematics}, 45(5):2647--2675, 2019.

\bibitem{Wang2002}
L.~Wang and M.~Xu.
\newblock Well-posedness of dual-phase-lagging heat conduction equation: higher
  dimensions.
\newblock {\em International Journal of Heat and Mass Transfer},
  45(5):1165--1171, 2002.

\bibitem{Wang2001}
L.~Wang, M.~Xu, and X.~Zhou.
\newblock Well-posedness and solution structure of dual-phase-lagging heat
  conduction.
\newblock {\em International Journal of Heat and Mass Transfer},
  44(9):1659--1669, 2001.

\bibitem{Ye2013}
H.~Ye, F.~Liu, I.~Turner, V.~Anh, and K.~Burrage.
\newblock Series expansion solutions for the multi-term time and space
  fractional partial differential equations in two- and three-dimensions.
\newblock {\em The European Physical Journal Special Topics},
  222(8):1901--1914, 2013.

\bibitem{Zeidler1990IIA}
E.~Zeidler.
\newblock {\em Nonlinear Functional Analysis and its Applications II/A: Linear
  Monotone Operators}.
\newblock Springer, 1990.

\bibitem{Zhao2012}
Z.~Zhao and C.~Li.
\newblock Fractional difference/finite element approximations for the
  time-space fractional telegraph equation.
\newblock {\em Appl. Math. Comput.}, 219(6):2975--2988, 2012.

\bibitem{Zheng2016}
M.~Zheng, F.~Liu, V.~Anh, and I.~Turner.
\newblock A high-order spectral method for the multi-term time-fractional
  diffusion equations.
\newblock {\em Appl. Math. Modelling}, 40(7-8):4970--4985, 2016.

\bibitem{Zheng2020b}
X.~Zheng and H.~Wang.
\newblock Wellposedness and smoothing properties of history-state-based
  variable-order time-fractional diffusion equations.
\newblock {\em Zeitschrift für angewandte Mathematik und Physik}, 71(1):34,
  2020.

\bibitem{Zhou2021}
Y.~Zhou and J.~W. He.
\newblock Well-posedness and regularity for fractional damped wave equations.
\newblock {\em Monatsh. Math.}, 194(2):425--458, 2021.

\end{thebibliography}
		\bibliographystyle{abbrv}
\end{document}